\documentclass{amsart}

\usepackage{amsmath}
\usepackage{amsfonts}
\usepackage{amssymb}
\usepackage{graphicx}
\usepackage{mathrsfs}
\usepackage{pb-diagram}
\usepackage{amscd}
\usepackage{hyperref}
\usepackage{url}
\usepackage{caption}
\usepackage{subcaption}
\usepackage{fancybox}
\usepackage{wrapfig}
\usepackage{color}
\usepackage[all]{xy}

\newtheorem{theorem}{Theorem}[section]
\newtheorem{lemma}[theorem]{Lemma}

\newtheorem{prop}[theorem]{Proposition}
\newtheorem{defn}[theorem]{Definition}

\newtheorem{remark}[theorem]{Remark}

\numberwithin{equation}{section}

\newcommand{\Z}{\mathbb{Z}}
\newcommand{\R}{\mathbb{R}}
\newcommand{\C}{\mathbb{C}}
\newcommand{\bS}{\mathbb{S}}

\newcommand{\CF}{\mathcal{F}}

\newcommand{\bL}{\mathbf{L}}
\newcommand{\BL}{\mathbb{L}}

\newcommand{\Hom}{\mathrm{Hom}}

\newcommand{\Fuk}{\mathrm{Fuk}}

\newcommand{\one}{\mathbf{1}}

\newcommand{\bP}{\mathbb{P}}

\newcommand{\cO}{\mathcal{O}}

\newcommand{\pt}{\mathrm{pt}}

\newcommand{\MF}{\mathrm{MF}}

\newcommand{\Coh}{\mathrm{Coh}}

\begin{document}

\title[Moduli of Lagrangians]{Moduli of Lagrangian immersions with formal deformations}

\author[Hong]{Hansol Hong}
\address{Center of Mathematical Sciences and Applications\\Harvard University}
\email{hhong@cmsa.fas.harvard.edu, hansol84@gmail.com}
\author[Lau]{Siu-Cheong Lau}
\address{Department of Mathematics and Statistics\\ Boston University}
\email{lau@math.bu.edu}
 
\begin{abstract}
We introduce a joint project with Cheol-Hyun Cho on the construction of quantum-corrected moduli of Lagrangian immersions.  The construction has important applications to mirror symmetry for pair-of-pants decompositions, SYZ and wall-crossing.  The key ingredient is Floer-theoretical gluing between local moduli spaces of Lagrangians with different topologies.

\end{abstract}

\maketitle

\section{Introduction}
Moduli theory for vector bundles has been very well established.  The groundbreaking discovery of Donaldson \cite{Don} and Uhlenbeck-Yau \cite{UY} built a deep connection between stability and canonical metric.

\begin{theorem}[\cite{Don,UY}]
A slope-semistable holomorphic vector bundle admits a Hermitian Yang-Mills metric.
\end{theorem}

Stability conditions and GIT were essential to the construction.  In general terms, Bridgeland \cite{Bridgeland} developed a general mathematical theory of stability conditions for triangulated categories based on Douglas' work \cite{Douglas} on $\Pi$-stability on D-branes in string theory.
Foundational techniques have been developed in \cite{Bridgeland-K3,Toda1,BMT} to construct stability conditions for derived categories of coherent sheaves.  One remarkable feature of the theory is that, if we vary our choice of a stability condition, then the moduli spaces of stable objects undergo birational changes \cite{Bridgeland-flop}, and they are related by Fourier-Mukai transforms.

We are working to develop a moduli theory for Lagrangians.  According to homological mirror symmetry \cite{Kont-HMS}, coherent sheaves on a complex manifold are mirror to Lagrangian submanifolds in its mirror symplectic manifold.  Thus there should be a mirror moduli theory for Lagrangian submanifolds.  

The following lists some ingredients that are new to the moduli theory for coherent sheaves.
\begin{enumerate}
\item {\bf Complexification.}  The classical moduli spaces are affine manifolds with singularities \cite{Hitchin,McLean}.  To get better compactifications,
we need to complexify the moduli spaces by considering flat connections or other boundary deformations on the Lagrangians.  Technically we need to work over the Novikov ring \cite{FOOO,FOOO-T}
\begin{align*}
\Lambda_+ =& \left\{\sum_{i=0}^\infty a_i T^{A_i} \mid A_i > 0 \textrm{ increase to } +\infty, a_i \in \C \right\},\\
\Lambda_0 = &\left\{\sum_{i=0}^\infty a_i T^{A_i} \mid A_i \geq 0 \textrm{ increase to } +\infty, a_i \in \C \right\},\\
\Lambda = &\left\{\sum_{i=0}^\infty a_i T^{A_i} \mid A_i \textrm{ increase to } +\infty, a_i \in \C \right\},\\
\Lambda_0^\times =& \, \C^\times \oplus \Lambda_+. \\
\end{align*}
They are equipped with the valuation function 
$$\mathrm{val}:\sum_{i=0}^\infty a_i T^{A_i} \mapsto A_0 \textrm{ and } \mathrm{val}(0) = +\infty.$$
The moduli are rigid analytic spaces.
\item {\bf Quantum corrections.}  The canonical complex structures need to be corrected using Lagrangian Floer theory \cite{FOOO}.  For moduli of SYZ fibers, the combinatorial structure of quantum corrections was deeply studied by Kontsevich-Soibelman \cite{KS-affine} and Gross-Siebert \cite{GS07}.
\item {\bf Landau-Ginzburg models.}  After compactification, the moduli are usually singular algebraic varieties.  
They can be described as critical loci of holomorphic functions which are known as Landau-Ginzburg models.  Orlov \cite{Orlov} found the theory of matrix factorizations for Landau-Ginzburg models.  More generally we may encounter noncommutative Landau-Ginzburg models \cite{Bocklandt, CHL2}.
\item {\bf Singular Lagrangians.}  Geometrically singular Lagrangians (or even isotropic skeletons) can appear as limits of smooth Lagrangians.  They are still not yet very well understood.  Kontsevich \cite{Kont-SG} proposed to study them using cosheaves of categories.  Nadler \cite{Nadler} is developing a study of arboreal singularities to classify them.  We need to understand what objects these singular Lagrangians correspond to in the Fukaya category.
\end{enumerate}

A particularly important class of examples is the moduli of fibers in a Lagrangian torus fibration.  Strominger-Yau-Zaslow \cite{SYZ} conjectured that torus duality produces mirror geometries.  Family Floer theory is the key ingredient to derive homological mirror symmetry in the SYZ setting.
The theory was proposed by Fukaya \cite{Fukaya-famFl} to study mirror symmetry by using Floer cohomologies of fibers of a Lagrangian torus fibration.

Tu \cite{Tu-reconstruction} took this approach to construct mirror spaces away from singular fibers.  Abouzaid \cite{Ab-famFl2} constructed family Floer functors for torus bundles and showed that the functor is fully faithful.  The famous Fukaya trick plays an important role, which relates Floer theories of different fibers by isotopies of $A_\infty$-algebras.  However singular Lagrangian fibers place a major difficulty to understand SYZ and family Floer theory.

Lagrangian immersions are the best singular Lagrangians in the sense that they still have a well-defined Floer theory by Akaho-Joyce \cite{AJ}.  In \cite{CHL,CHKL,CHL2}, we constructed mirror geometries as the local moduli of Lagrangian immersions and derived homological mirror symmetry.

In the study of SYZ, the immersed two-sphere in (a) of Figure \ref{fig:mirror_glue} is particularly interesting, since it (or the product with a torus) is the main source of wall-crossing phenomenons.  

Furthermore, it is hoped that any object in the derived Fukaya category can be given as a Lagrangian immersion; in particular any singular Lagrangian is isomorphic (as an object in the Fukaya category) to an immersion, so that we do not need to tackle singular Lagrangians other than immersions at all!  For instance, Lagrangian immersions in the pair-of-pants were used by Seidel \cite{Seidel-g2} and Sheridan \cite{Sheridan11,Sheridan-CY} to prove homological mirror symmetry for Fermat-type hypersurfaces (see (b) of Figure \ref{fig:mirror_glue} for such an immersion in dimension 1).  

\begin{figure}[htb!]
    \includegraphics[scale=0.5]{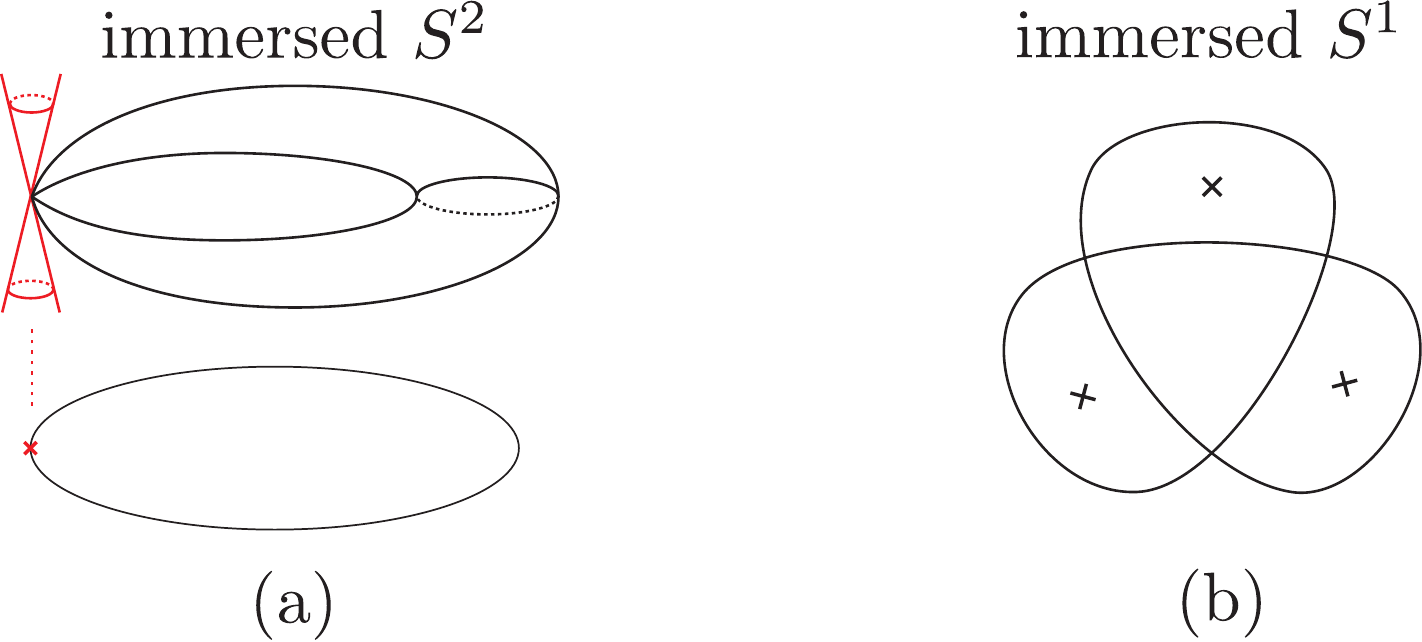}
    \caption{Immersed spheres}
		\label{fig:mirror_glue}
\end{figure}

In this expository article, we introduce ongoing joint works with Cheol-Hyun Cho, in which we use Floer-theoretical gluing methods to construct  moduli of Lagrangian immersions which are not necessarily tori.  There are many occasions that non-tori play an important role, for instance in pair-of-pants decompositions or geometries related to knots.  An interesting feature is that topologies of Lagrangians can jump within the same moduli space.  We cannot use the Fukaya trick because the Lagrangians are no longer diffeomorphic to each other.  

There is a technical advantage of having Lagrangian immersions in our moduli.  Namely, immersed sectors provide bigger formal deformation spaces than flat connections.  The deformation parameters associated to immersed sectors lie in $\Lambda_+$ (or $\Lambda_0$ in good cases), while flat connections are only $\C^\times$-valued.  (Lagrangian Floer theory does not allow deformations by $\Lambda^\times$-connections since it would destroy the energy filtration.  We will consider them as `pseudo-deformations'.)  Intuitively $\Lambda_+$ consists of infinitely many layers of $\C^\times$ in different energy levels.  Thus the formal deformation space of an immersion typically covers a local family of Lagrangians (which can even be singular).  

The main idea is to consider \emph{pseudo-isomorphisms between Lagrangian immersions in the moduli and obtain gluing information from cocycle conditions on pseudo-isomorphisms}.  This is explained in Section \ref{sec:gluing}.  We introduce two applications, namely mirror construction by pair-of-pants decompositions, and generic wall-crossing in SYZ.  

We will present the main ideas and leave the details in separate papers.  To begin with, let's revisit the most basic example $\bP^1$, which already contains some essential ideas used in our later developments.  Fukaya-Oh-Ohta-Ono \cite{FOOO-T,FOOO-T2,FOOO-T3} gave a deep study on the Lagrangian Floer theory of toric manifolds and proved closed-string mirror symmetry.

\subsection*{Example: the two-sphere}
Consider $\bP^1$ equipped with the spherical area form (as a symplectic form) and the meromorphic top-form $dz/z$.  It has an $\bS^1$-symmetry.  The $\bS^1$-moment map $\bP^1 \to [0,1]$ gives a special Lagrangian fibration.

The classical moduli space of special Lagrangian fibers \cite{McLean,Hitchin} is given by $(0,1)$.  The induced affine structure on $(0,1)$ is trivial.  The classical moduli does not capture symplectic topologies of Lagrangians, for instance the non-displaceability of fibers.

We can complexify the moduli space by decorating the special Lagrangian fibers with flat $U(1)$-connections.  We have the set
$$\{(\textrm{fiber, flat $U(1)$-connection})\} = (0,1) \times \bS^1.$$  
This is the well-known mirror space \cite{HV,CL}.

Instead of taking the set of special Lagrangian fibers decorated by flat $U(1)$-connections, we can take
\begin{equation} \label{eq:moduli}
\{(\textrm{stable Lagrangian in fiber class, flat $\Lambda_0^\times$-connection})\} / \textrm{(non-zero) Quasi-isom.}
\end{equation}
A `stable Lagrangian in fiber class' is a Lagrangian which is quasi-isomorphic to a special Lagrangian fiber.  Here we consider quasi-isomorphisms in the Fukaya category, see Section \ref{subsec:isominfuk}.  Moreover we enlarge the deformation space for each Lagrangian by taking $\Lambda_0^\times = \C^\times \oplus \Lambda_+$.  (Flat $\C^\times$ connections was introduced by \cite{Cho-non-unitary} to find non-displaceable Lagrangians.)

The above is simply $(0,1) \times \Lambda_0^\times$ as a set.  Different fibers do not intersect with each other and hence there is no morphism between them.  This is not good enough for understanding the rigid-analytic structure of the moduli.

Consider the further enlargement
\begin{equation} \label{eq:enlarge}
\{(\textrm{stable Lagrangian in fiber class, flat $\Lambda^\times$-connection})\} / \textrm{(non-zero) Pseudo-isom.}
\end{equation} 
Note that we cannot allow flat $\Lambda^\times$-connections in the Fukaya category, since it would destroy the energy filtration and Novikov convergence of the category.  Lagrangians equipped with flat $\Lambda^\times$-connections are not really objects in the Fukaya category.  However it is important to consider them for the rigid analytic structure of the moduli space.  We call these to be `pseudo-deformations'.  We define `pseudo-isomorphisms' between them formally the same as quasi-isomorphisms, namely morphisms (linear combinations of transverse intersection points over $\Lambda$ satisfying the cocycle condition) that have inverses up to homotopy.  (See Definition \ref{def:pdeform} and \ref{def:piso}.)

By allowing pseudo-deformations and pseudo-isomorphisms, there are relations between different fibers.  For instance consider the fibers $L_1$ and $L_2$ as in Figure \ref{fig:sphere}.  Take a rotation of $L_1$ (which is a Hamiltonian perturbation) and get $L_1'$.  Since $L_1'$ is isomorphic to $L_1$, it is stable.  

\begin{figure}[htb!]
    \includegraphics[scale=0.5]{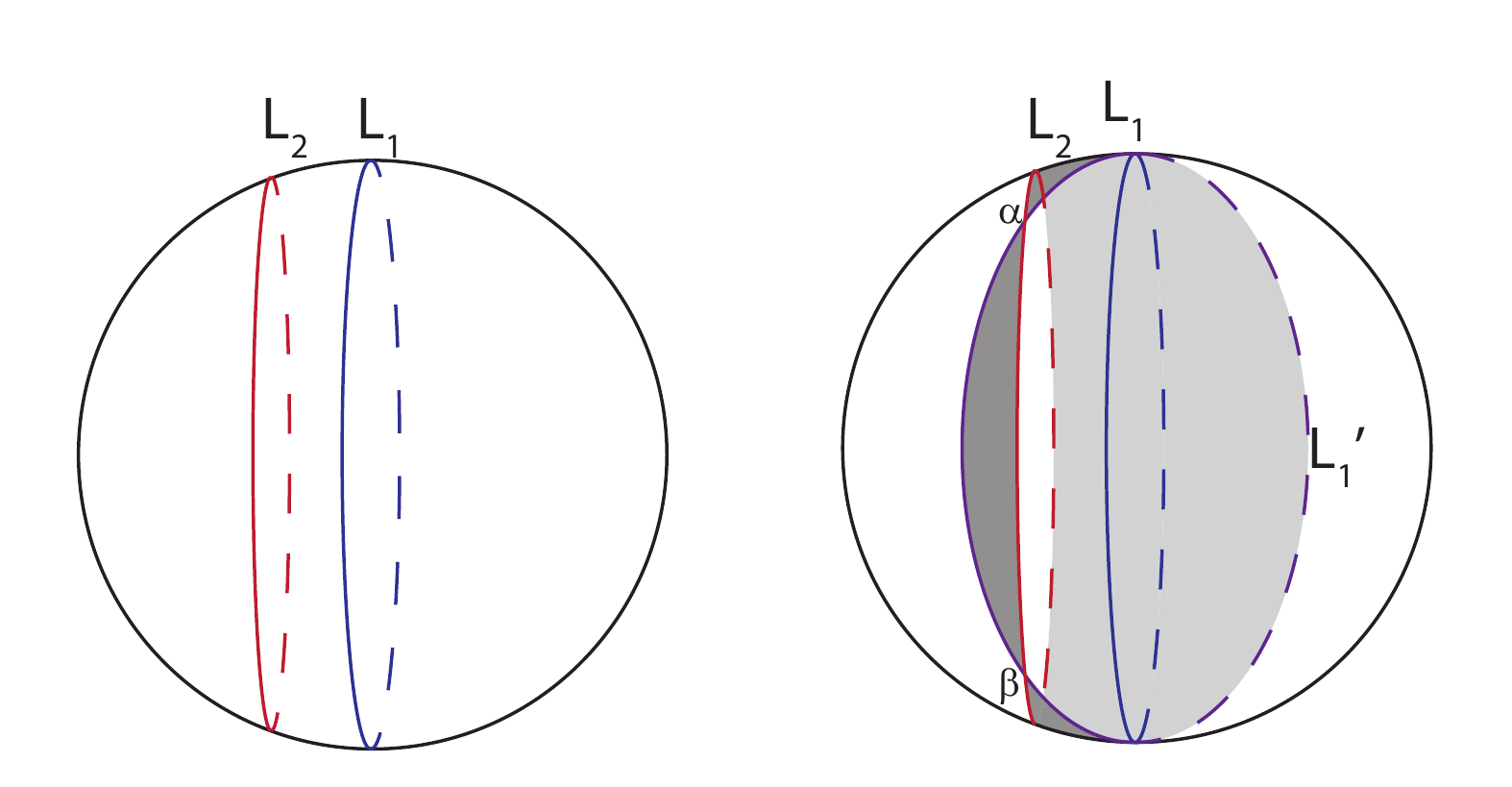}
    \caption{Stable Lagrangians in $\bP^1$.}
		\label{fig:sphere}
\end{figure}

$L_1'$ intersects $L_2$ at two intersection points $(\alpha,\beta)$.  There are two strips bounded by $L_1'$ and $L_2$.  Note that these two strips have different areas.
For a flat connection $\nabla_t$ on $L_1'$, we have 
$$m_1^{\nabla_t} (\alpha) = T^c (T^A - t) \beta$$ 
where $c$ is a constant and $A>0$ is the area of the cylinder bounded between $L_1$ and $L_2$.  

$T^A - t \not= 0$ if $t$ is only allowed to be $\Lambda_0^\times$-valued.  Thus there is no morphism between $(L_1',\nabla_t)$ and $L_2$.  This is as expected since $L_2$ (equipped with any decoration) is a trivial object in the Fukaya category, while $L_1'$ is not.  $(L_1',\nabla_t)$ and $L_2$ are different elements in Equation \eqref{eq:moduli}.

We can set $t=T^A$ if we allow it to be $\Lambda^\times$-valued! Then $(\alpha,\beta)$ provides a pseudo-isomorphism between $(L_1',\nabla_t)$ and $L_2$.  Thus $(L_1',\nabla_t)$ (for $t=T^A \in \Lambda^\times$) and $L_2$ are the same element in Equation \eqref{eq:enlarge}.

We see that any fibers are pseudo-isomorphic to $(L_1,\nabla_t)$ for some $t \in \Lambda^\times$ (even though they are not intersecting with each other).  The expression \eqref{eq:enlarge} equals to $\Lambda^\times$ which has a rigid analytic structure.  Expression \eqref{eq:moduli} is a subset of Expression \ref{eq:enlarge} and hence equipped with the induced topology.  Indeed it is $\Lambda^\times_{0 < \mathrm{val < 1}}$ which equals to $(0,1) \times \Lambda_0^\times$ as a set.

Expression \eqref{eq:moduli} is not yet the moduli.  We have a superpotential $W$ coming from the $m_0$-term of each element in \eqref{eq:moduli}.  In this case $W = t + T^{A_S} t^{-1}$ where the two terms correspond to the left and the right discs bounded by a fiber (and $A_S$ is the area of the sphere).  The moduli is the Landau-Ginzburg model 
$(\Lambda^\times_{0 < \mathrm{val < 1}},W)$, whose critical locus consists of two points corresponding to two flat connections on the equator which is non-displaceable.  This was deeply studied in \cite{FOOO-T}.

\section{Deformation space of a Lagrangian immersion}

We review basic Lagrangian Floer theory, in particular the formal deformation of Lagrangian submanifolds following \cite{FOOO}. We focus on two specific types of formal deformation, one by flat line bundles over Lagrangian tori and the other by immersed generators for Lagrangian immersions. We will also briefly explain immersed Lagrangian Floer theory by Akaho-Joyce \cite{AJ}.

\subsection{Review on Lagrangian Floer theory: formal deformation}

Let $(X,\omega)$ be a symplectic manifold (possibly noncompact) and consider a compact Lagrangian $L$ in $X$. By \cite{FOOO}, we have a gapped filtered $A_\infty$-algebra $\left(CF(L,L),\{m_k\}_{k \geq 0} \right)$ possibly with a nontrivial curvature $m_0$. For a degree odd (or degree one when $L$ is graded) element $b$, one can deform this $A_\infty$-structure as follows.
\begin{equation}\label{eqn:fdefb}
m_k^{b, \cdots, b} (x_1, \cdots, x_k):= \sum_{\substack{i_1,\cdots, i_{k+1} \\ i_1 + \cdots + i_{k+1} = l}} m_{k+l} (\overbrace{b, \cdots, b}^{i_1} ,x_1, b, \cdots, b , x_k , \overbrace{b, \cdots, b}^{i_{k+1}}).
\end{equation}
It is easy to see that $\{m_k^b\}$ defines an $A_\infty$-algebra \cite{FOOO}.

We are interested in $b$ which staistifes the following condition so that $(L,b)$ defines an object of the Fukaya category.

\begin{defn} $b$ is called a  Maurer-Cartan element if
$$m_0^b= m_0 (e^b) =m_0(1) + \sum_{k=1}^{\infty} m_k (b, \cdots,b ) = 0.$$
In this case $(L,b)$ is said to be unobstruced.

It is called weak Maurer-Cartan if the right hand side is a (Novikov) constant multiple of $\one_{L} := PD[L]$, in which case $(L,b)$ is said to be weakly unobstructed.
\end{defn}

We denote the space of such $b$'s by $\mathcal{MC}(L)$ for Maurer-Cartan elements, and by $\mathcal{MC}_{weak} (L)$ for weak Maurer-Cartan elements.  Note that deformations by different elements in $\mathcal{MC}(L)$ could lead to isomorphic objects.  We shall identify them in the construction of moduli spaces.  Isomorphisms are discussed in Section \ref{sec:gluing}.

$\mathcal{MC} (L)$ can be interpreted as a (formal) deformation space of $L$.  For instance, in the case of toric manifolds, $\mathcal{MC}_{weak} (L)$ agrees with the first cohomology of $L$ which encodes geometric deformation of $L$ (see \cite{FOOO-T}).

In case there is no danger of confusion, we will sometimes call $\mathcal{MC} (L)$ (or $\mathcal{MC}_{weak} (L)$) deformation space or moduli space, omitting ``formal" in front of them.

\subsection{Deformation by $\C^\times$ flat line bundles on $L$}

Instead of using elements in $CF(L,L)$, one can vary the holonomy of a flat line bundle over a Lagrangian $L$ to deform $L$ (as an object of Fukaya category). More precisely, if we equip $L$ with a  line bundle $E$ with a $\C^\times$-flat connection $\nabla$, the $A_\infty$-operations on $CF(L,L)$ are deformed as follows:
$$m_k^{(L,\nabla)} (x_1, \cdots, x_k) =  \sum_{\beta \in \pi_2(M,L)} \left( hol_{\partial \beta} \nabla \right) m_{k,\beta} (x_1, \cdots, x_k) T^{\omega(\beta)}$$
where $m_{k,\beta}$ is the contribution from holomorphic disks in class $\beta$ to the original $m_k$-operator on $CF(L,L)$. By considering all possible holonomies, we obtain a (formal) moduli of objects in $\Fuk(M)$ isomorphic to $\left(\C^\times\right)^{\dim_\R H_1 (L;\R)}$, which is nothing but the space of all $\C^\times$ flat line bundles on $L$ modulo equivalence. One can slightly enlarge this space by considering $\Lambda^\times$-line bundle, but the holonomy is still required to have valuation zero (i.e. it should start with nontrivial complex number without $T$) so we only have a slice of $\left(\Lambda^\times \right)^{\dim_\R H_1 (L;\R)}$.

In what follows, we will mainly consider Lagrangian tori equipped with a $\C^\times$-flat line bundle.
Furthermore, we will always choose a special representative of $\C^\times$-line bundles whose connection behaves like a delta function. More precisely, for $L \cong \R^n / \Z^n$, we fix (oriented) hyper-tori $H_i = \epsilon_i + \R \langle e_i \rangle$ for $\epsilon_i \in \R/\Z$ so that the parallel transport over a path $\gamma$ is given by multiplying $z_i^{\pm}$ whenever $\gamma$ runs across $H_i$ where the sign in the exponent is determined by the parity of the intersection $\gamma \cap H_i$. (See \cite{CHLtoric} for more details.)

\subsection{Immersed Floer theory}
For an immersed Lagrangian $L$, we consider the boundary deformation induced by self-intersections. First, let us review Floer theory on $L$ developed by \cite{AJ}. We assume that $L$ has transversal intersection only. Let $\iota : L \to M$ be such a Lagrangian immersion. The Floer complex $CF(L,L)$ is generated by cochains on
$$L \times_{\iota} L = \{ (x,y) \in L \times L \mid \iota(x) = \iota(y)\}.$$
Notice that each self-intersection point of $\iota$ gives two distinct generators, which are points in
the off-diagonal component of $L \times_{\iota} L$. They have the following geometric meaning in Floer theory.
Let $A$ and $B$ be two local branches of $L$ meeting at a self-intersection point $X$ of $\iota$. Then the corresponding generators, say $X$ and $\bar{X}$, of $CF(L,L)$ describe the jumps of the boundary of a holomorphic disk from $A$ to $B$ or from $B$ to $A$, when we travel along the boundary in positive orientation.

If $X_1 \cdots, X_l$ denote odd-degree (degree one for $L$ graded) generators from self-intersection points, we use the linear combination $b:=\sum x_i X_i$ to formally deform the Lagrangian $L$ as in \eqref{eqn:fdefb} where $x_i$'s are a priori free variables. Namely $x_i$'s are taken from the free algebra generated by themselves. Here, we fix the convention for $A_\infty$-operations involving free variables by setting
$$m_k (x_{i_1} X_{i_1} ,\cdots, x_{i_k} X_{i_k} ):= x_{i_k} \cdots x_{i_1} m_k (X_{i_1}, \cdots, X_{i_k}),$$
i.e. we pull out the coefficients from the back.

Unlike $\C^\times$-flat connections, not every such linear combinations define legitimate deformation, as the Maurer-Cartan equation $\sum_k m_k (b,\cdots,b)$ may impose nontrivial relations among free variables $x_i$'s. 
Therefore the corresponding deformation space is the Spec of the algebra generated by $x_i$'s modulo relations from the MC-equation. Note that this algebra is in general noncommutative, though commutative variables $x_i$ solve MC-equation in  most of examples we will consider in this article.

\subsection{An example: pair-of-pants}\label{subsec:expop}
Consider the Lagrangian immersion $L$ from $S^1$ constructed by Seidel \cite{Seidel-g2} in a pair-of-pants.  See Figure \ref{fig:Seidel_Lag}.  There are three transverse immersed points, giving the immersed generators $X,Y,Z$ in odd degree and $\bar{X},\bar{Y},\bar{Z}$ in even degree.  The Floer complex is $\CF(L,L) = \mathrm{Span}\{\mathbf{1},X,Y,Z,\bar{X},\bar{Y},\bar{Z},\mathrm{pt}\}$ as a vector space.

\begin{figure}[htb!]
    \includegraphics[scale=0.4]{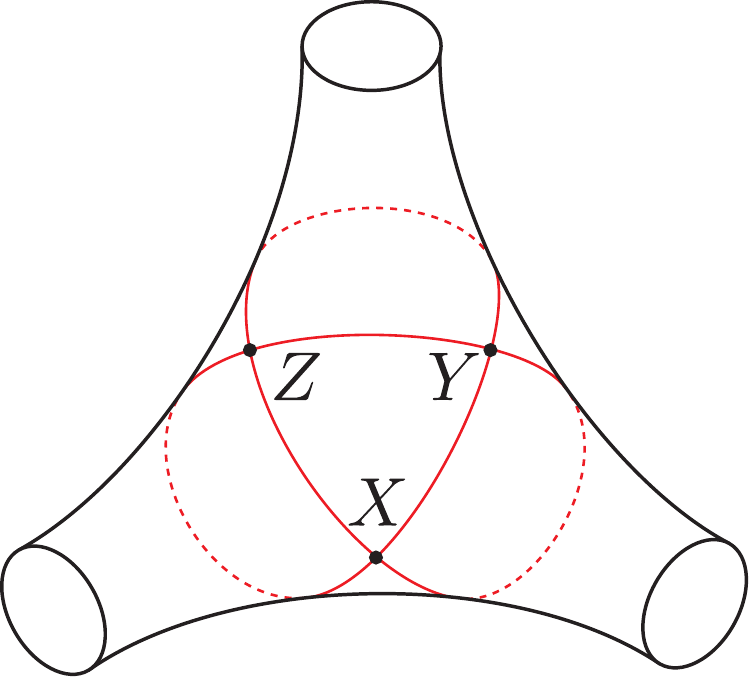}
    \caption{Seidel Lagrangian in a pair-of-pants}\label{fig:Seidel_Lag}
\end{figure} 

The following lemma is essential to make sense of the deformation theory (over $\C$).  It is due to the fact that the areas of the two triangles bounded by $L$ are the same, and their contributions to $m_0^b$ (for $b=xX + yY + zZ$) cancel with each other (when we choose a non-trivial spin structure on $L$).

\begin{lemma}
The formal odd-degree deformations $b=xX + yY + zZ$ are weakly unobstructed for $x,y,z \in \C$. i.e. $\mathcal{MC}_{weak} (L) \cong \C^3$. 
\end{lemma}

We have $m_0^b = W \cdot \one_L$ where $W = xyz$ and $\one_L$ is the unit (namely the fundamental class).  In conclusion the moduli space is given by the Landau-Ginzburg model $(\C^3, xyz)$.  In Section \ref{sec:app} we consider a pair-of-pants decomposition and glue these local spaces together to form a global moduli.

\begin{remark}
The pair-of-pants can be compactified to an orbifold $\bP^1_{a,b,c}$.  In \cite{CHL,CHKL} we used the Seidel Lagrangian $L$ to construct and compute the mirror, and derived homological mirror symmetry.  In an ongoing work with Amorim we are proving closed-string mirror symmetry along this line.  In \cite{LZ} it was shown that the mirror superpotential $W$ of elliptic orbifolds ($\frac{1}{a}+\frac{1}{b}+\frac{1}{c}=1$) has interesting modular properties.
\end{remark}

\section{Gluing of formal deformation spaces} \label{sec:gluing}

We will glue formal moduli spaces of Lagrangians making use of explicit isomorphisms in the Fukaya category. We first recall the notion of isomorphisms in an $A_\infty$-category, and explain how it can be used in our geometric setting in which we have two Lagrangians whose formal deformations are isomorphic over a certain subset of formal moduli spaces. At the end of the section, we describe a homotopy model for the B-side category for the glued formal moduli space following \cite{BB2013}.

\subsection{Isomorphisms between objects of the Fukaya category}\label{subsec:isominfuk}

First, we introduce the notion of isomorphisms between two objects in an $A_\infty$-category.

\begin{defn}
Let $L$ and $L'$ be two objects in an $A_\infty$-category $\mathcal{C}$. $\alpha \in \Hom_{\mathcal{C}} (L,L')$ is called a quasi-isomorphism if $m_1(\alpha)=0$ and there exists $\beta \in \Hom_{\mathcal{C}} (L',L)$ such that 
$$m_2(\alpha,\beta) = \one_{L} + m_1 (\gamma) \quad \mbox{and} \quad m_2(\beta,\alpha) = \one_{L'} + m_1 (\gamma')$$
for some $\gamma$ and $\gamma'$. If $m_1(\gamma)=m_1(\gamma')=0$, then $\alpha$ (and $\beta$) is called a strict-isomorphism (or simply an isomorphism).
\end{defn}

One can easily check that quasi-isomorphisms define an equivalence relation on the set of objects in $\mathcal{C}$. In particular, the composition of two quasi-isomorphisms (i.e. $m_2$ of the two) is also a quasi-isomorphism.

We next examine the gluing of formal moduli spaces for two different (possibly immersed) Lagrangians $L$ and $L'$ in a symplectic manifold $M$. For $b \in U \subset \mathcal{MC} (L)$ and $b' \in  U' \subset \mathcal{MC} (L')$, suppose that there is a diffeomorphism $f : U \to U'$ and a quasi-isomorphism
$$ \alpha : (L,b) \to (L',f(b))$$
between objects in $\Fuk (M)$. In general $\alpha$ may depend on $b$ as well, but here we assume that $\alpha$ is fixed element in the vector space $\oplus_{x \in L \cap L'} \Lambda \langle x \rangle$ for simplicity. 

When $L$ and $L'$ are weakly unobstructed with nontrivial potentials 
$$ W: \mathcal{MC}_{weak} (L) \to \Lambda,\quad W': \mathcal{MC}_{weak} (L') \to \Lambda,$$
we further require that $W'(f(b)) = W(b)$ in order to glue two resulting Landau-Ginzburg models $(\mathcal{MC}_{weak} ( L),W)$ and $(\mathcal{MC}_{weak} (L'), W')$. In fact, this condition is necessary to have a quasi-isomorphsm $\alpha : (L,b) \to (L', f(b))$ since the Floer differential $m_1^{b,b'}$ on $CF((L,b),(L',b'))$ for $b'=f(b)$ squares to be zero only under such an assumption, since
$$ \left(m_1^{b,b'} \right)^2 = W(b) - W'(f(b)).$$

\subsection{Pseudo-deformations and Pseudo-isomorphisms} \label{sec:pseudo}
To understand the relation between two nearby Lagrangians in the same moduli space, we have to consider pseudo-deformations and pseudo-isomorphisms.  They are not valid in Fukaya category.  However they provide a way to talk about Lagrangians close to $L$ (but not isomorphic to $L$) by formal pseudo-deformations on $L$.  Using this notion a neighborhood of $L$ in the moduli can be identified as an open subset of $\Lambda^n$ and is endowed with a rigid analytic structure.

\begin{defn} \label{def:pdeform}
A pseudo-deformation $b$ of a Lagrangian immersion $L$ is either a flat $\Lambda^\times$-connection on $L$, or a linear combination of degree-one immersed generators of $L$ over the Novikov field $\Lambda$, which satisfies that $m_k^{b,\ldots,b}$ converge in Novikov sense for all $k \geq 0$.
\end{defn}

Intuitively the Novikov convergence above means that the formal deformation is small enough so that we can talk about its image under the exponential map.  Note that the above condition concerns $L$ itself only but not any other Lagrangians.  In the filtered Fukaya category we only allow formal deformations by flat $\Lambda_0^\times$-connections or linear combinations of degree-one immersed generators with coefficients in $\Lambda_{>0}$ (or $\Lambda_{\geq 0}$ in better situations) which ensures Novikov convergence of the whole Fukaya category.

\begin{defn} \label{def:piso}
Let $b,b'$ be weakly unobstructed pseudo-deformations on $L,L'$ respectively.  $\alpha \in \Hom (L,L')$ is called a pseudo-isomorphism from $(L,b)$ to $(L',b')$ if the $A_\infty$-operations $m_k^{\mathbf{b}_0,\ldots, \mathbf{b}_k}$ for all $k\geq 0$ (defined in the same way as in the case of formal deformations) converge in Novikov sense where $\mathbf{b}_i$ can be either $b$ or $b'$, $m_1^{b,b'}(\alpha)=0$ and there exists $\beta \in \Hom (L',L)$ (with $m_1^{b',b}(\beta)=0$) such that 
$$m_2^{b,b',b}(\alpha,\beta) = \one_{L} + m_1^{b,b} (\gamma) \quad \mbox{and} \quad m_2^{b',b,b'} (\beta,\alpha) = \one_{L'} + m_1^{b',b'} (\gamma')$$
for some $\gamma$ and $\gamma'$.

$(L,b)$ and $(L',b')$ are said to be pseudo-isomorphic to each other if there is a chain of pseudo-isomorphisms $\alpha_1: (L,b) \to (L_1,b_1)$, $\alpha_2: (L_1,b_1) \to (L_2,b_2)$, \ldots, $\alpha_k: (L_{k-1},b_{k-1}) \to (L',b')$.
\end{defn}

Conceptually the Novikov convergence above means $L,L'$ are close enough to each other in the moduli so that we can compare their pseudo-deformations.

\begin{remark}
As in the example of $\bP^1$ in the Introduction, two pseudo-isomorphisms may not compose to a pseudo-isomorphism since the $m_k$ operations involving the three Lagrangians with pseudo-deformations may not converge.  Thus we need a chain of pseudo-isomorphisms in the above definition.
\end{remark}

\subsection{Construction of $A_\infty$-functor}

Given two formal moduli spaces and their gluing data as in the previous section, we have a B-model category that encodes complex information of glued moduli whose construction goes as follows. Let $\mathcal{D}_1$ and $\mathcal{D}_2$ be two dg-categories and suppose there are dg-functors 
$$\phi_i : \mathcal{D}_i \to \mathcal{E} \quad i=1,2$$
that land on another dg-category $\mathcal{E}$. Given these data, we have the following new dg-cateogory which models the gluing of $\mathcal{D}_1$ and $\mathcal{D}_2$ over $\mathcal{E}$.

\begin{defn}
The homotopy fiber product $\mathcal{D}_1 \times_{\mathcal{E}} \mathcal{D}_2$ is a dg-category whose objects are tuples of the form
$$\left(D_1, D_2, \phi_1 (D_1) \stackrel{\sigma}{\to} \phi_2 (D_2) \right)$$
where $\sigma$ is required to descend to an isomorphism on the cohomology level. A morphism between two objects 
is given as a tuple
\begin{equation}\label{eqn:morhptydg}
\begin{array}{l}
\Hom_{\mathcal{D}_1 \times_{\mathcal{E}} \mathcal{D}_2}^i ( (D_1,D_2, \sigma ), (D_1',D_2', \sigma') ) \\
= \Hom_{\mathcal{D}_1}^i (D_1, D_1' ) \oplus \Hom_{\mathcal{D}_2}^i (D_2, D_2') \oplus \Hom_{\mathcal{E}}^{i-1} (\phi_1 (D_1), \phi_2 (D_2')).
\end{array}
\end{equation}
\end{defn}

A morphism $(\mu_1,\mu_2, \tau)$ in \eqref{eqn:morhptydg} is closed under the differential if and only if both $\mu_1$ and $\mu_2$ are closed morphisms and $\tau$ is a homotopy between $\sigma' \circ \phi_1(\mu_1)$ and $ \phi_2 (\mu_2) \circ \sigma$.
\begin{equation*}
\xymatrix{ \phi_1(D_1) \ar[r]^{\sigma} \ar[d]_{ \phi_1(\mu_1)} \ar[dr]^{\tau}& \phi_2 (D_2) \ar[d]^{\phi_2(\mu_2)}\\
\phi_1 (D_1') \ar[r]_{\sigma'} & \phi_2 (D_2')
}
\end{equation*}
We will not spell out the dg-structure on $\mathcal{D}_1 \times_{\mathcal{E}} \mathcal{D}_2$ in further details. See \cite{BB2013} for the precise definition.

In our geometric situation, $\mathcal{D}_i$ is the category of coherent sheaves on $\mathcal{MC} (\BL_i)$ for two fixed Lagrangians $\BL_i$ $i=1,2$, or the matrix factorization category $\MF(W_i)$ when we deal with weak Maurer-Cartan elements and the resulting Landau-Ginzburg model $(\mathcal{MC}_{weak} (\BL_i), W_i)$. For an isomorphism
$$ f: U_1 \subset \mathcal{MC} (\BL_1) \to U_2 \subset \mathcal{MC} (\BL_2)$$
(or between subsets of $\mathcal{MC}_{weak} (\BL_i)$), which underlies a quasi-isomorphism
\begin{equation}\label{eqn:pisom12}
 \alpha : (\BL_1, b_1) \to (\BL_2, f(b_1)),
\end{equation}
$\mathcal{E}$ is either (dg-enhanced) $D^b \Coh (U_1)$ or $\MF(U_1, W|_{U_1})$ depending on the situation. 
The functors $\phi_i:\mathcal{D}_i \to \mathcal{E}$ are naturally induced by the restriction maps in this case.

\begin{theorem} We have a natural $A_\infty$-functor
$$ \Fuk(M) \to \mathcal{D}_1 \times_{\mathcal{E}} \mathcal{D}_2$$
where $\mathcal{D}_i$ and $\mathcal{E}$ are given as in the explanation above. 
\end{theorem}

Let us fix an inverse $\beta \in CF(\BL_2, \BL_1)$ of the quasi-isomorphism $\alpha$ \eqref{eqn:pisom12}. The functor sends a Lagrangian $L \in \Fuk(M)$ to a tuple
$$ ( CF((\BL_1,b_1), L), CF((\BL_2, b_2), L), \sigma_\beta: CF((\BL_1,b_1), L)\to f^\ast CF((\BL_2, f(b_1)), L))   $$
where $b_1$ and $b_2$ vary over the corresponding formal deformation spaces for the first two components, and over $U$ and $U'$ in the third component. Note that both $CF((\BL_1,b_1), L)$ and $f^\ast CF((\BL_2, f(b_1)), L))$ are bundles over $U_1$. $\sigma_\beta$ is a map (homotopy equivalence in $A_\infty$-languat) is defined by
$$p \in \BL_1 \cap L \mapsto \sum_{i,j \geq 0} \pm m_{i+j+2} (\overbrace{b',\cdots,b'}^{i}, \beta, \overbrace{b,\cdots,b}^{j},p ). $$
Precise sign rules and the construction of higher components of the functor (as well as the proof of the theorem) will appear in the separate paper.
Throughout the article, we will only study the (formal) moduli spaces of Lagrangians themselves, and will not discuss these functors further.

We remark that one can also consider gluing of more than two charts using the categorical model for such a situation as in, for e.g., \cite{Wei2016}.

\section{Application I: mirror construction for pair-of-pants decompositions} \label{sec:app}

In this section we will use the gluing method given in the last section to construct a moduli space of Lagrangian immersions in a pair-of-pants decompsition.  It was briefly introduced in \cite{L16} and \cite{L17}.  Here we make a more precise formulation.
We shall focus only on punctured Riemann surfaces which already capture the essential ingredients.  We will complete the detail in higher dimensions in a future work.  

Homological mirror symmetry was proved for punctured Riemann surfaces due to the works of Abouzaid-Auroux-Efimov-Katzarkov-Orlov \cite{AAEKO}, Bocklandt \cite{Bocklandt}, Pascaleff-Sibilla \cite{PS2016} and Heather Lee \cite{HLee}.  In this article we will put our focus on the construction of the mirror as a moduli space of Lagrangian immersions.  We will use this construction to derive HMS in a separate paper.

\begin{figure}[htb!]
    \includegraphics[scale=0.35]{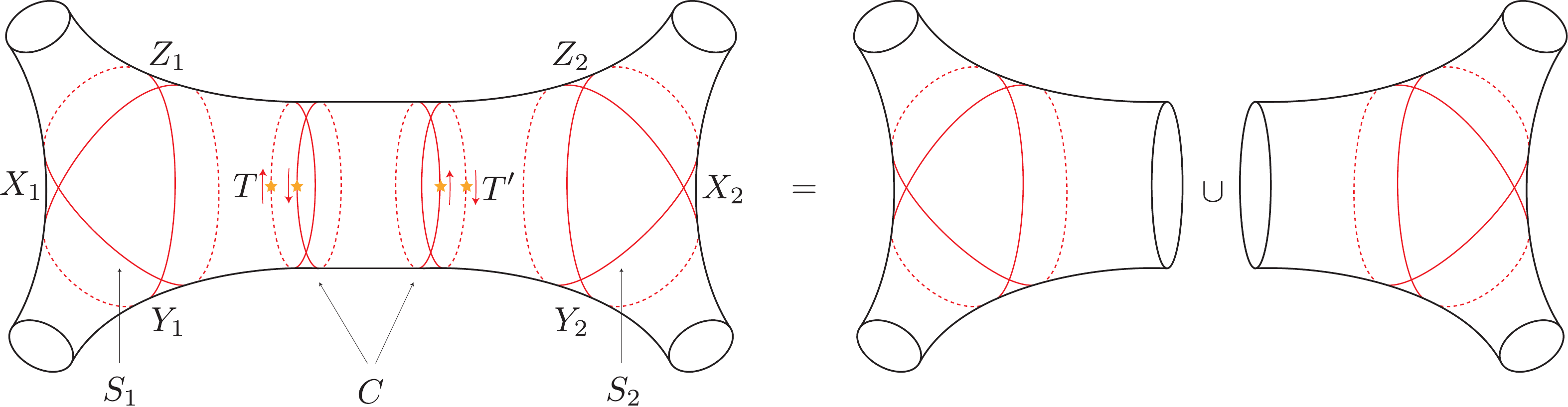}
    \caption{A pair-of-pants decomposition of the four-punctured sphere and immersed Lagrangians.}
		\label{fig:pp-decomp}
\end{figure}

For simplicity let's consider the four-punctured sphere as shown above.  The construction for more general pair-of-pants decompositions follows from the same method.  (We will clarify a delicate point for general pair-of-pants decompositions in Section \ref{sec:Novikov}).  Let's work over $\C$ at this stage, although later on we shall see that in general it is necessary to work over the Novikov ring $\Lambda_0$ (see Section \ref{sec:paradox}).

Recall from Section \ref{subsec:expop} that the deformation spaces of a Seidel Lagrangian is given by $(\C^3,W)$.
We need to glue the deformation spaces $(\C^3,W)$ of the two Seidel Lagrangians $S_1$ and $S_2$.  

To change from $S_1$ to $S_2$, 
there are two main processes: (partial) smoothing of $S_1$ (or $S_2$) to a double circle $C$, and gauge change of flat connections on $C$. (See Figure \ref{fig:pp-decomp}.)
The gluing we need is
$$ (\C^3,W^{S_1}) \overset{\small\textrm{smoothing}}\longleftrightarrow (\C^\times \times \C^2,W^{(C,T)}) \overset{\small\textrm{gauge change}}\longleftrightarrow (\C^\times \times \C^2,W^{(C,T')}) \overset{\small\textrm{smoothing}}\longleftrightarrow (\C^3,W^{S_2}). $$
(The notation $T$ and $T'$ are explained below.)

\subsection{Choices of gauge change}
First we consider gauge change which is more standard.  There is a vanishing sphere in the smoothing $C$ corresponding to the immersed point $X_1$ of $S_1$.  In this dimension it is simply the union of two points.  

Put a flat $\C^\times$ connection on $C$, which is acting by $t \in \C^\times$ when passing through the two points (in a chosen direction shown in Figure \ref{fig:pp-decomp}).  The union of the two points with the prescribed normal orientation is called a gauge cycle, and it is denoted by $T$.  Similarly $C$ can be obtained by smoothing $S_2$.  Denote the corresponding gauge cycle by $T'$.

The gauge cycles $T$ and $T'$ are different.  Changing the gauge cycles corresponds to gauge change of flat connections.  Let's take a homotopy from $T$ to $T'$ by moving the gauge cycles along the Lagrangian immersion $C$ (without jumping across branches).  In the process the gauge cycles pass through immersed points $Y_0$ or $Z_0$.

When a gauge cycle moves across an immersed point (say $Y_0$) of $C$, we can prove that the Floer theories before and after the move are isomorphic under a non-trivial change of coordinates as follows.

\begin{lemma}
Suppose a component in the gauge cycle $T$ passes through the immersed point $Y_0$ (and denote the new gauge cycle by $\tilde{T}$).  Denote the objects before and after the move by $(C,\nabla^{tT})$ and $(C,\nabla^{\tilde{t}\tilde{T}})$ where $\nabla^{tT}$ and $\nabla^{\tilde{t}\tilde{T}}$ are flat connections with the corresponding gauge.  Then the objects $(C,\nabla^{tT},y_0Y_0+z_0Z_0)$ and $(C,\nabla^{\tilde{t}\tilde{T}},\tilde{y}_0Y_0+\tilde{z}_0Z_0)$ are isomorphic via $\tilde{t}=t$, $\tilde{y}_0 = ty_0$ and $\tilde{z}_0=z_0$.
\end{lemma}

Note that there are infinitely many different homotopies between $T$ and $T'$.  For instance we can move one of the gauge points through the immersed point $Y$, and the other through $Z$.  Then the gluing is $t'=t^{-1}$, $y_0' = ty_0$ and $z_0'=tz_0$.  (We have $t'=t^{-1}$ instead of $t'=t$, due to the difference of normal orientations of $T$ and $T'$.)

In general 
$$t'=t^{-1}, y_0' = t^ay_0 \textrm{ and }z_0'=t^bz_0$$
for fixed $a,b \in \Z$ with $a+b=2$.
Thus different choices \emph{result in different models for the moduli}.  Nevertheless the critical loci of the superpotentials are isomorphic and the Landau-Ginzburg models are equivalent.

\subsection{Smoothing}  
Now consider the gluing between the deformation spaces of $S_1$ and $C$.
They are given by $(x_1,y_1,z_1) \in \C^3$ and $(t,y_0,z_0) \in \C^\times \times \C^2$ respectively.  (We only use $\C$-valued deformations for the moment.)

The superpotentials are $x_1y_1z_1$ and $ty_0z_0$ for $C$ and $S_1$ respectively.  Intuitively to match them, we should put $x_1=t,y=y_0,z=z_0$.  Below we justify this by Floer theory.  In this example, the gluing map is very simple and we can obtain it by guess; the following gives a more systematic way to find the gluing map.

We claim that $(S_1,x_1X_1+y_1Y_1+z_1Z_1)$ and $(C,tT+y_0Y_0+z_0Z_0)$ are `pseudo'-isomorphic to each other under a suitable gluing map analogous to above.  Note that $S_1$ and $C$ are not intersecting with each other at all, and so they can never be isomorphic! We will come back to this point in Section \ref{sec:paradox}.

The trick is to deform $S_1$ to $S_1^x$ which intersects $C$ at eight points $a_i,b_i,c_i,d_i$ for $i=1,2$ as in the figure.  We use cocycle conditions to deduce the gluing between $S_1^x$ and $C$.

\begin{figure}[htb!]
    \includegraphics[scale=0.25]{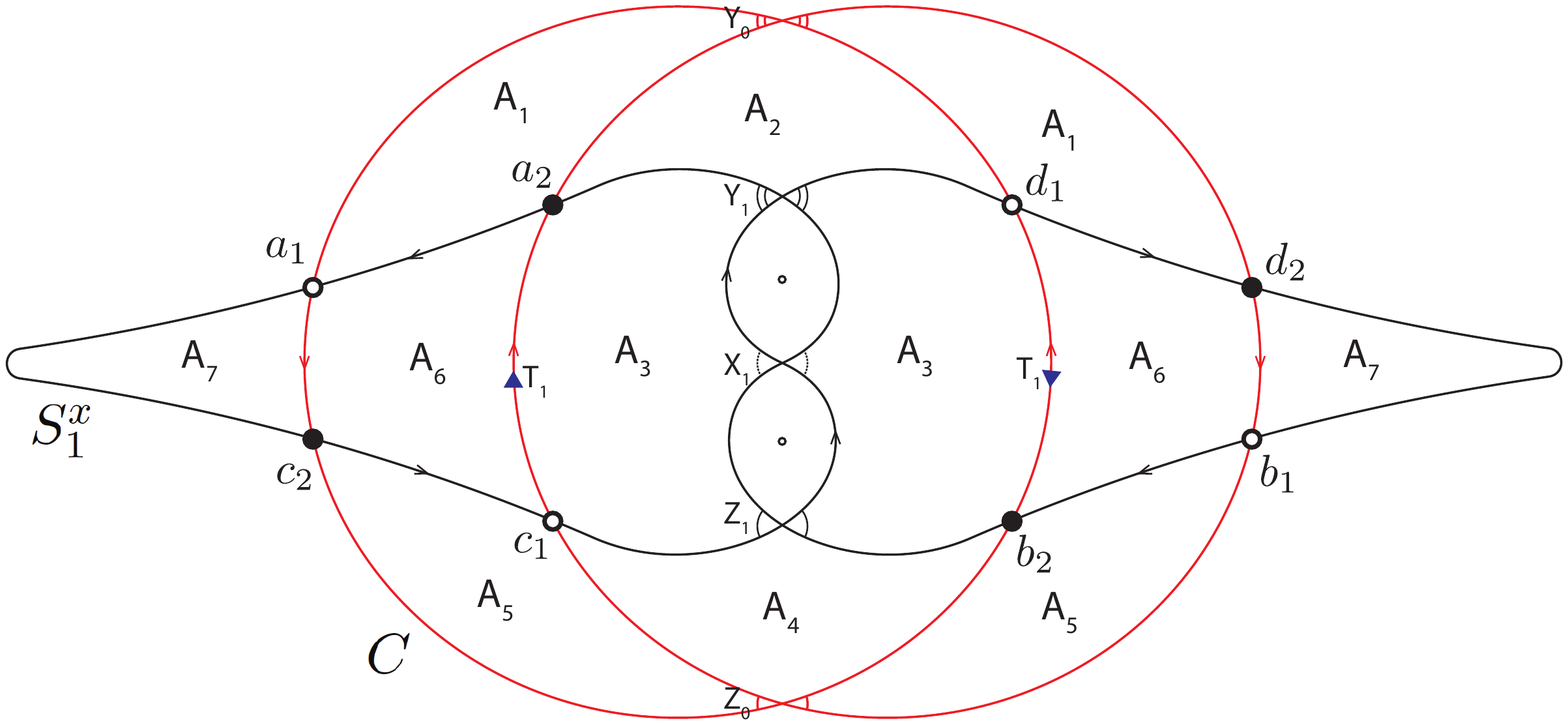}
    \caption{The deformed Seidel Lagrangian $S_1^x$ and the double-circle $C$.}
\end{figure}

Consider $a_1+b_1 \in CF((C,\nabla^{tT},y_0Y_0+z_0Z_0),(S_1^x,x_1X_1+y_1Y_1+z_1Z_1))$ and $c_2+d_2 \in CF((S_1^x,x_1X_1+y_1Y_1+z_1Z_1),(C,\nabla^{tT},y_0Y_0+z_0Z_0))$.  Consider cocycle conditions on $a_1+b_1$ and $c_2+d_2$.  First we check that the cocycle conditions are satisfied.  Then we show that $(a_1+b_1,c_2+d_2)$ gives isomorphisms between the two objects.

\begin{prop} \label{prop:cocycle_4punc}
For a suitable choice of $S_1^x$, we have 
$$m_1^{((C,\nabla^{tT},y_0Y_0+z_0Z_0),(S_1^x,x_1X_1+y_1Y_1+z_1Z_1))}(a_1+b_1) = 0$$ and $$m_1^{((S_1^x,x_1X_1+y_1Y_1+z_1Z_1),(C,\nabla^{tT},y_0Y_0+z_0Z_0))}(c_2+d_2)=0$$ if and only if $x_1=t,y_1=y_0,z_1=z_0$.
\end{prop}

It is obtained by direct counting of strips.  See Figure \ref{fig:puncsph-discs}.

\begin{figure}[htb!]
    \includegraphics[scale=0.4]{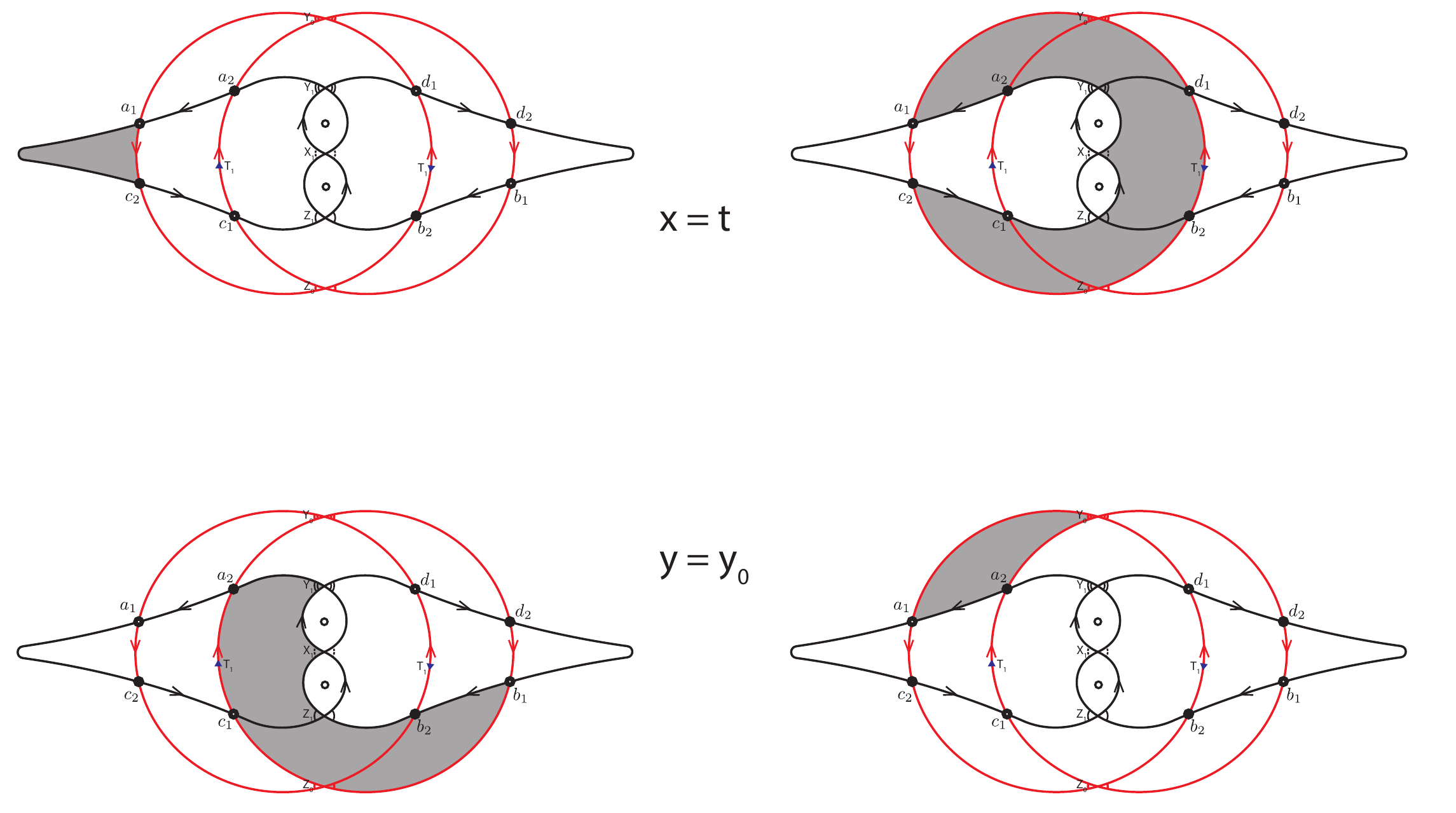}
    \caption{Holomorphic strips used in Proposition \ref{prop:cocycle_4punc}.}
		\label{fig:puncsph-discs}
\end{figure}

Similarly we can verify the following.

\begin{prop}
Under the same choice of $S_1^x$ as in Proposition \ref{prop:cocycle_4punc}, $m_2(a_1+b_1,c_2+d_2)=\one_{S_1^x}$ and $m_2(c_2+d_2,a_1+b_1)=\one_{C}$ if $x_1=t,y_1=y_0,z_1=z_0$.
\end{prop}

In the same way we have the gluing map between $(C,\nabla^{t'T'},y'_0Y'_0+z'_0Z'_0)$ and $(S_2^{x},x_2X_2+y_2Y_2+z_2Z_2)$.  

In conclusion, we have
$$ (\C^3,W^{S_1}) \overset{\small\textrm{smoothing}}\longleftarrow (\C^\times \times \C^2,W^{(C,T)}) \overset{\small\textrm{gauge change}}\longleftrightarrow (\C^\times \times \C^2,W^{(C,T')}) \overset{\small\textrm{smoothing}}\longrightarrow (\C^3,W^{S_2}) $$
where the first map is $x_1=t,y_1=y_0,z_1=z_0$, the second map is $t'=t^{-1}$, $y_0' = t^ay_0$ and $z_0'=t^bz_0$, the third map is $x_2=t',y_2=y_0',z_2=z_0'$.  

The maps automatically preserve the superpotentials $W^{S_i}=x_iy_iz_i$, $W^{(C,T)}=ty_0z_0$ and $W^{(C,T')}=t'y'_0z'_0$.
The resulting moduli space is 
$$(\cO(-a)\oplus\cO(-b),W) \textrm{ with }a+b=2 $$
which is the Landau-Ginzburg mirror of the four-punctured sphere.

\subsection{A paradox} \label{sec:paradox} 
We have shown in the previous section that $(S_1^x,x_1X_1+y_1Y_1+z_1Z_1)$ is isomorphic to $(C,\nabla^{tT},y_0Y_0+z_0Z_0)$ by setting $x_1=t,y_1=y_0,z_1=z_0$.  On the other hand, it is not difficult to see that $S_1$ is isomorphic to $S_1^x$, see Figure \ref{fig:S1-S1x} (where $(\alpha,\beta)$ serves as an isomorphism.)  

Thus $(S_1,x_1X_1+y_1Y_1+z_1Z_1)$ is isomorphic to $(C,\nabla^{tT},y_0Y_0+z_0Z_0)$.  But $S_1$ and $C$ never intersect and hence there is no morphism between them, a contradiction!

\begin{figure}[htb!]
    \includegraphics[scale=0.25]{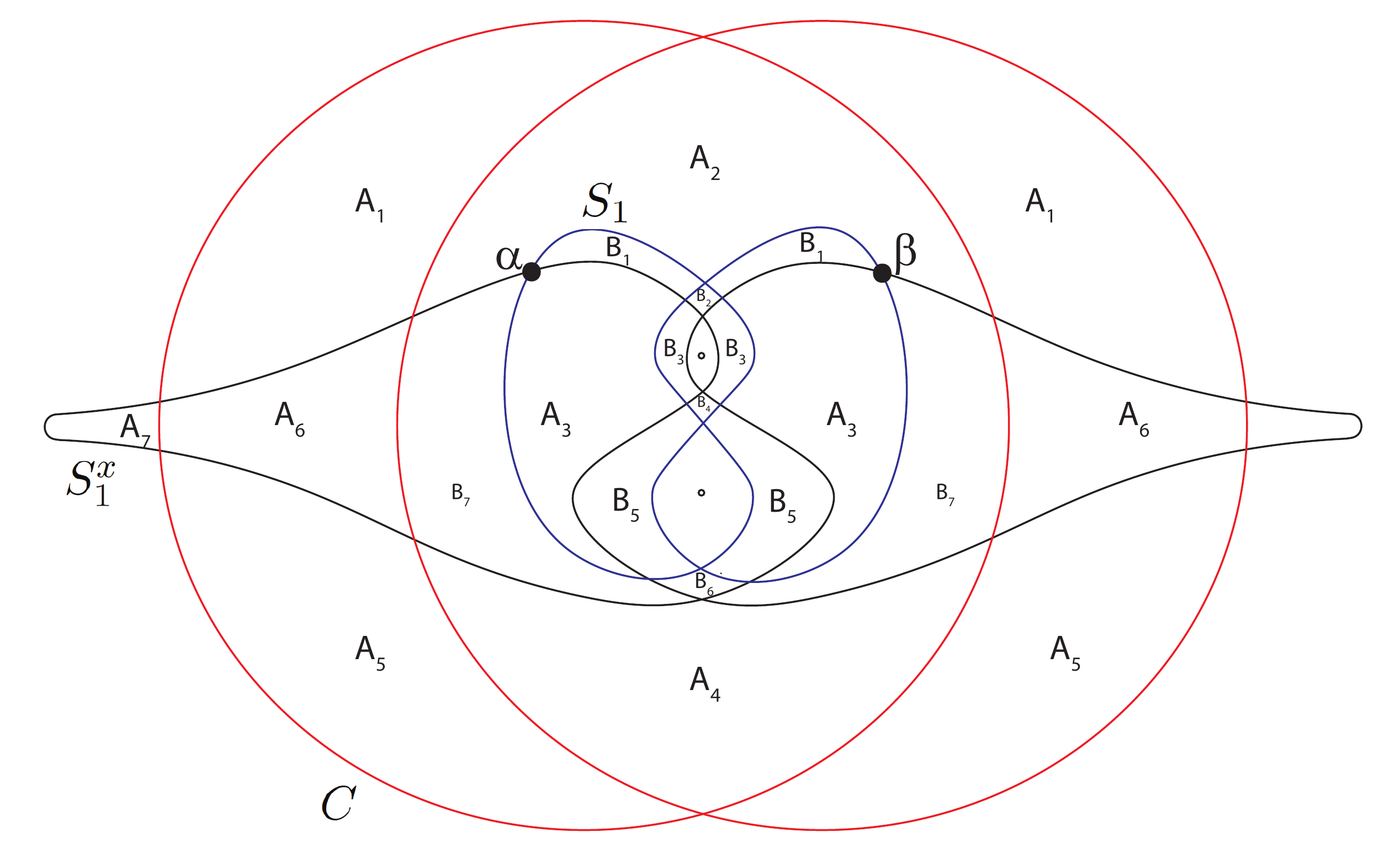}
    \caption{The immersions $S_1$, $S_1^x$, and $C$.  $A_i$ denote the areas bounded by $S_1^x$ and $C$, and $B_i$ denote the areas bounded by $S_1$ and $S_1^x$.  The figure is topological and does not represent the exact areas.}
		\label{fig:S1-S1x}
\end{figure}

To resolve this, we need to take a closer look at areas of holomorphic discs, the relation between $S_1$ and $S_1^x$, and the relation between $S_1^x$ and $C$.
Following Fukaya-Oh-Ohta-Ono, we use the Novikov ring 
$$\Lambda_0 = \left\{\sum_{i=0}^\infty a_i T^{A_i}: a_i \in \C, 0 \leq A_0 \leq A_1 \leq \ldots \textrm{ increasing to } \infty \right\}$$ 
to filter deformations into different energy levels.   Analogously one has the Novikov field $\Lambda$ (where $A_0 \geq 0$ is not required) and $\Lambda_+$ (where $A_0 > 0$ is required).

First let's go back to the relation between $(S_1^x,x_1X_1)$ and $(C,\nabla^{tT})$.  (For simplicity we take $y_1=z_1=y_0=z_0=0$.)  
The following is a more precise version of Proposition \ref{prop:cocycle_4punc}.

\begin{prop}
For a general choice of $S_1^x$, $(a_1+b_1,c_2+d_2)$ provides an isomorphism between $(C,\nabla^{tT})$ and $(S_1^x,x_1X_1)$
if and only if 
\begin{equation} \label{eq:x_1-t}
t=T^{A_1+\ldots+A_5-A_7} x_1
\end{equation}
where $A_i$ denote the areas bounded by $S_1^x$ and $C$ shown in Figure \ref{fig:S1-S1x}.  
\end{prop}
In Proposition \ref{prop:cocycle_4punc} we take $A_1+\ldots+A_5=A_7$.

Before we have restricted $x_1 \in \C$ for simplicity.  Indeed we can take $x_1 \in \Lambda_0$, which makes the deformation space of $S_1^x$ much bigger.  

For the Lagrangian $C$, since $\nabla^{tT}$ is a flat connection, we can only allow $t \in \Lambda_{\mathrm{val}=0}= \C^\times \oplus \Lambda_+$ in order to have the Fukaya category well-defined.  Thus there is an overlapping region for 
Equation \eqref{eq:x_1-t} if and only if 
$$A_1+\ldots+A_5 \leq A_7.$$

Thus for arbitrary $S_1^x$ the isomorphism may not exist.  (For instance it does not exist for $S_1$.)
The following proposition gives a summary of the area criterion.

\begin{prop}
There exists $x \in \lambda_0$ and $t \in \Lambda_{\mathrm{val}=0}$ such that $(S_1^x,x_1X_1)$ is isomorphic to $(C,\nabla^{tT})$ if and only if $A_1+\ldots+A_5 \leq A_7$.  
\end{prop}

Now consider the relation between $S_1$ and $S_1^x$.  We have the following.

\begin{prop} \label{prop:S1-S1x}
$(\alpha,\beta)$ provides an isomorphism between $(S_1^x,x_1X_1)$ and $(S_1,x_1'X_1')$ if and only if
\begin{equation} \label{eq:x1-x1'}
x_1=T^{B_7+B_6+B_5-(B_3+B_2+B_1)} x_1'
\end{equation}
where $B_i$ denote the areas bounded by $S_1^x$ and $S_1$ shown in Figure \ref{fig:S1-S1x}.  
\end{prop}

Let's go back to the situation of Proposition \ref{prop:cocycle_4punc} where $S_1^x$ is chosen such that $A_1+\ldots+A_5=A_7$ and hence $x_1=t$.  Note that $B_7 > A_7$ and $B_1+B_2+B_3 < A_3 + A_5$.  Hence $B_7+B_6+B_5-(B_3+B_2+B_1) > 0$.  Then Equation \eqref{eq:x1-x1'} forces $\mathrm{val}(x_1) > 0$.  Thus $(S_1^x,x_1X_1)$ is isomorphic to some $(S_1,x_1'X_1')$ only if $\mathrm{val}(x_1) > 0$.  On the other hand, $(S_1^x,x_1X_1)$ is isomorphic to some $(C,\nabla^{tT})$ only if $\mathrm{val}(x_1) = 0$.

The two regions $\mathrm{val}(x_1) = 0$ and $\mathrm{val}(x_1) > 0$ are disjoint.  As a result, there is no overlapping region where $(S_1,x_1'X_1') \cong (S_1^x,x_1X_1) \cong (C,\nabla^{tT})$ simultaneously.  This resolves the paradox.

\subsection{Moduli over the Novikov field} \label{sec:Novikov}

For the four-punctured sphere with two pair-of-pants in the decomposition, we can choose $S_1^x$ and $S_2^x$ as above in the two pair-of-pants such that the gluing $x_1=t$ and $x_2=t'$ does not involve the Novikov parameter $T$.  

However suppose there is one more pair-of-pants in the decomposition.  Then there are two `necks' containing the exact double-circles $C_1$ and $C_2$.  In the pair-of-pants adjacent to both necks, we have a Seidel Lagrangian $S_1^x$ which satisfies $(S_1^x,x_1X_1) \cong (C_1,\nabla^{tT})$ under $x_1=t \in \Lambda_{\mathrm{val}=0}$, and another Seidel Lagrangian $S_1^y$ which satisfies $(S_1^y,y_1' Y_1') \cong (C_2,\nabla^{t'T'})$ under $y_1'=t' \in \Lambda_{\mathrm{val}=0}$.

Unavoidably the gluing between deformation spaces of $S_1$ and $S_1'$ involves the Novikov parameter as in Proposition \ref{prop:S1-S1x}.  Thus we obtain a moduli space over the Novikov field $\Lambda$.  Figure \ref{fig:Lambda-mod} gives a schematic picture of the moduli space.

\begin{figure}[htb!]
    \includegraphics[scale=0.25]{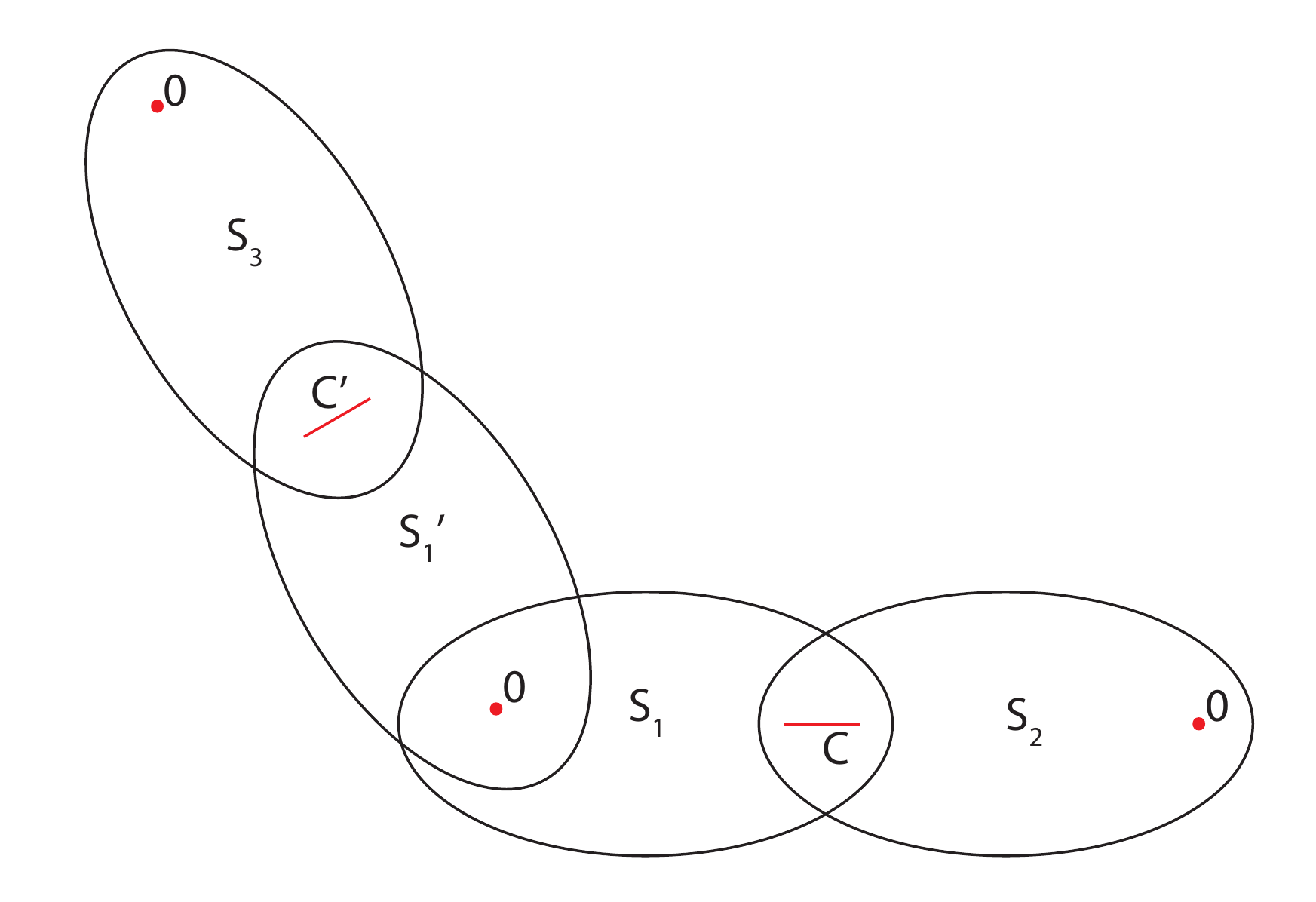}
    \caption{The moduli space over $\Lambda$.}
		\label{fig:Lambda-mod}
\end{figure}



\subsection{Changing stability and flop}
Let's briefly discuss stability conditions.
Given a punctured Riemann surface, we can take different choices of pair-of-pants decompositions.  They correspond to different choices of quadratic differentials.  Then we have different set of stable Lagrangians; the corresponding moduli spaces are related by birational changes.

In this article we only consider geometric stability (rather than categorical stability) which is defined by a quadratic top-form (namely a section of $(\bigwedge^n T^*)^{\otimes 2}$).  We define a special Lagrangian (of phase zero) as a smooth map $\iota: \hat{L}^n \to M^{2n}$ whose differential is injective everywhere except at zeros of the quadratic top-form, with $\iota^* \omega = 0$ and $\iota^* \mathrm{Re} \Omega = 0$.  By a domain-dependent Hamiltonian isotopy the map $\iota$ is made to avoid the zeros of the quadratic top-form such that its differential is everywhere injective, and it only self-intersects at transverse points.  Then any Lagrangian which is quasi-isomorphic to such $\iota$ is said to be stable.

Take the four-punctured sphere as an example.  Figure \ref{fig:flop} depicts two different pair-of-pants decompositions and the corresponding stable Lagrangians.  The corresponding moduli spaces are related by a flop.  

\begin{figure}[htb!]
    \includegraphics[scale=0.35]{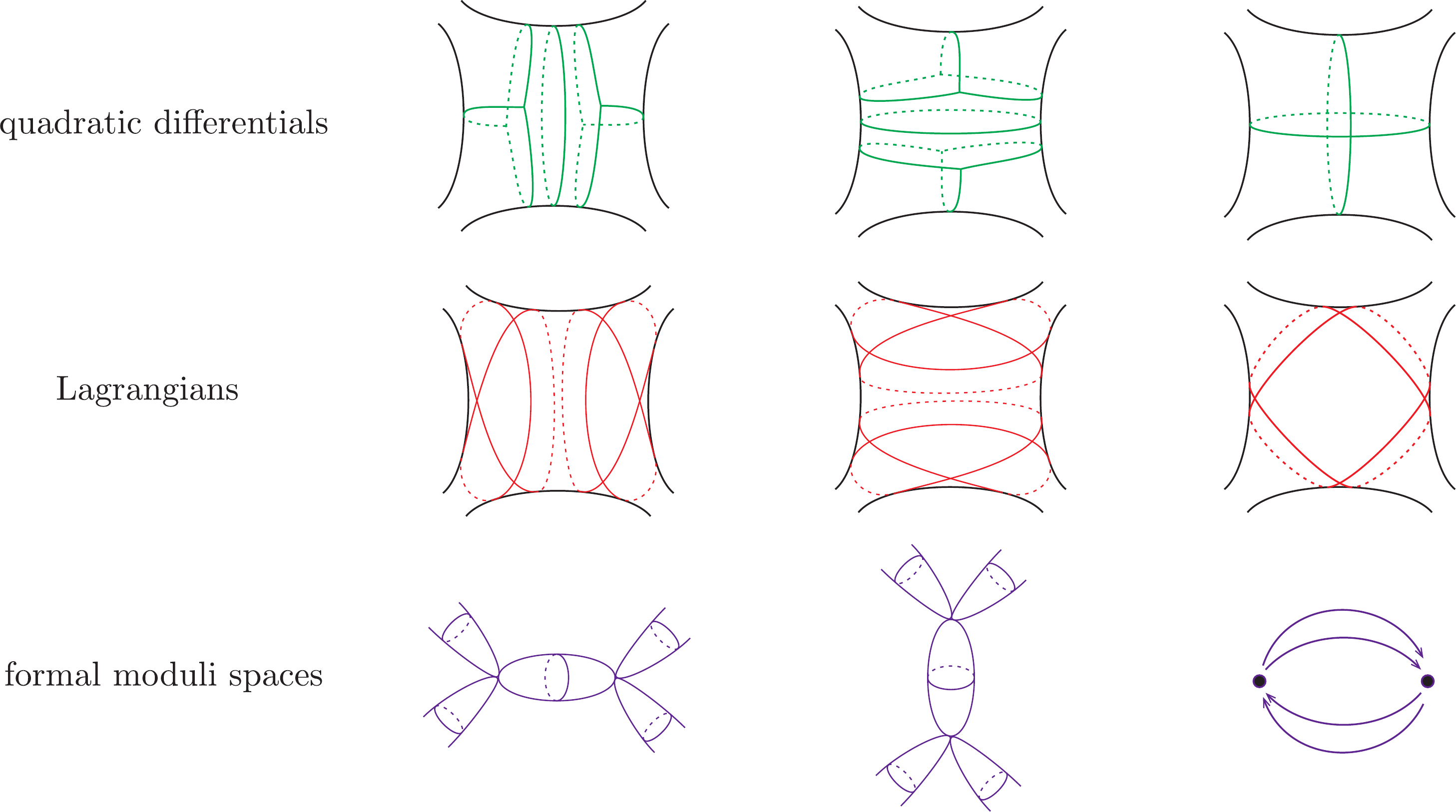}
    \caption{Different choices of quadratic differentials leads to different moduli related by flops.}
		\label{fig:flop}
\end{figure}

We can even merge two simple zeros into a double zero of the quadratic differential.  The corresponding flow trajectories and stable Lagrangians are depicted on the right of Figure \ref{fig:flop}.  We have a union of two circles playing the role of the Seidel Lagrangian.  It was studied in our previous work \cite{CHL2}.  The resulting moduli space is the non-commutative resolution of the conifold corresponding to a quiver (together with a superpotential).  

In \cite{FHLY} we studied ``flop" on a Lagrangian fibration for $T^* \bS^3$, which results in the Atiyah flop on the mirror resolved conifold.  If we take the union of two certain $S^3$'s in $T^* \bS^3$ as a reference Lagrangian, then we produced the non-commutative resolution of the conifold.  
It is a three-dimensional analog of the above example.

Stability conditions for punctured Riemann surfaces were studied by Haiden-Katzarkov-Kontsevich \cite{HKK}.  We wish to understand the relation with the work of \cite{HKK} later.

\section{Aplication II: immersed two-sphere and wall-crossing} \label{sec:2dim}
The two-dimensional immersed sphere given in (a) of Figure \ref{fig:mirror_glue} leads to wall-crossing phenomenons for SYZ Lagrangian fibrations.
Fukaya studied this immersion and the relation with the Gross-Siebert program in his talks.  

Here we use the gluing method introduced before to understand the relations between the local deformations of a Chekanov torus, the immersed sphere, and a Clifford torus.  The gluing between Chekanov and Clifford tori was computed in Seidel's lecture notes \cite[Prop. 11.8]{Seidel-lect} and is a special case of Pascaleff-Tonkonog \cite{PT}.  Thus we will focus on the gluing between the tori and the immersed sphere.  

One important thing is, we need to deform the Chekanov and Clifford fibers to make all three intersecting with each other appropriately.  Otherwise the fibers (of a Lagrangian fibration) do not intersect with each other and there is no morphism between them.

The immersed sphere has exactly one transverse immersed point, which gives two degree-one generators $U$ and $V$.  $U$ and $V$ can be understood as jumpings from one branch to the other, see the arrows in Figure \ref{fig:gluing}.  We work with Morse model and pearl trajectories to compute the morphisms between the immersed sphere and the tori.

\subsection{A local model}

Consider the Hamiltonian $S^1$-action on $\C^2$ defined by $(a,b) \mapsto (e^{i\theta} a, e^{-i\theta} b)$.  The moment map for this $S^1$-action is given by $(a,b) \mapsto |a|^2 - |b|^2.$  The symplectic reduction map is $f: \C^2 \to \C$ given by $f(a,b)=ab$.

We have the Lagrangian fibration $(|ab-\epsilon|,|a|^2 - |b|^2)$ on $\C^2 - \{ab = \epsilon\}$ for a fixed $\epsilon > 0$.  It is special with respect to the holomorphic volume form $\Omega = \frac{da \wedge db}{ab-\epsilon}$.  The fibers are denoted as $T_{r,\lambda}$ for $r > 0$ and $\lambda \in \R$.  SYZ and the wall-crossing phenomenon for this example was studied by Auroux \cite{Auroux07}, which serves as the basis of SYZ for more general cases \cite{Auroux09, CLL,AAK}.

There is exactly one singular fiber $\BL=T_{\epsilon,0}$, which is the two-dimensional immersed sphere given in (a) of Figure \ref{fig:mirror_glue}.  We will focus on this local model in this section.  The most generic wall-crossing in higher dimensions can be understood by considering the product of this Lagrangian fibration with the trivial torus bundle $(\C^\times)^k \to \R^k$.

\subsection{Deformation spaces}
Let $\BL_1$, $\BL_2 \subset \C^2 - \{ab = \epsilon\}$ be the Lagrangian tori which are the intersections of $\{|a|^2=|b|^2\}$ and the $f$-preimages of the circles in Figure \ref{fig:gluing}.  They are Lagrangian isotopic to a Chekanov fiber $T_{r_1,\lambda=0}$ and a Clifford fiber $T_{r_2,\lambda=0}$ respectively where $r_1<\epsilon<r_2$.

We deform $\BL_i$ (for $i=1,2$) formally by flat $\Lambda_0^\times = \C^\times \oplus \Lambda_+$ connections.  (We can allow flat $\Lambda_0^\times$ connections since it does not harm the convergence of $m_k$-operations.  Their inverses live in the same space $\C^\times \oplus \Lambda_+$, and hence no negative power of $T$ appears in $m_k$-operations.)  The connections are parametrized by $(x,y)$ and $(x',y')$ for $\BL_1,\BL_2$ respectively.

We fix the gauge by choosing gauge hyper-tori as depicted in Figure \ref{fig:gluing}.  The parallel transport is given by multiplying $y$ (for a path in $\BL_1$) or $y'$ (for a path in $\BL_2$) when a path goes across the corresponding hyper-torus.  For the fiber direction holonomy ($x$ or $x'$), the associated hyper-tori are given in Figure \ref{fig:dalpha} for $x$).  (See \cite{CHLtoric} for more details for gauge hypertori.)

\begin{figure}[h]
\begin{center}
\includegraphics[scale=0.4]{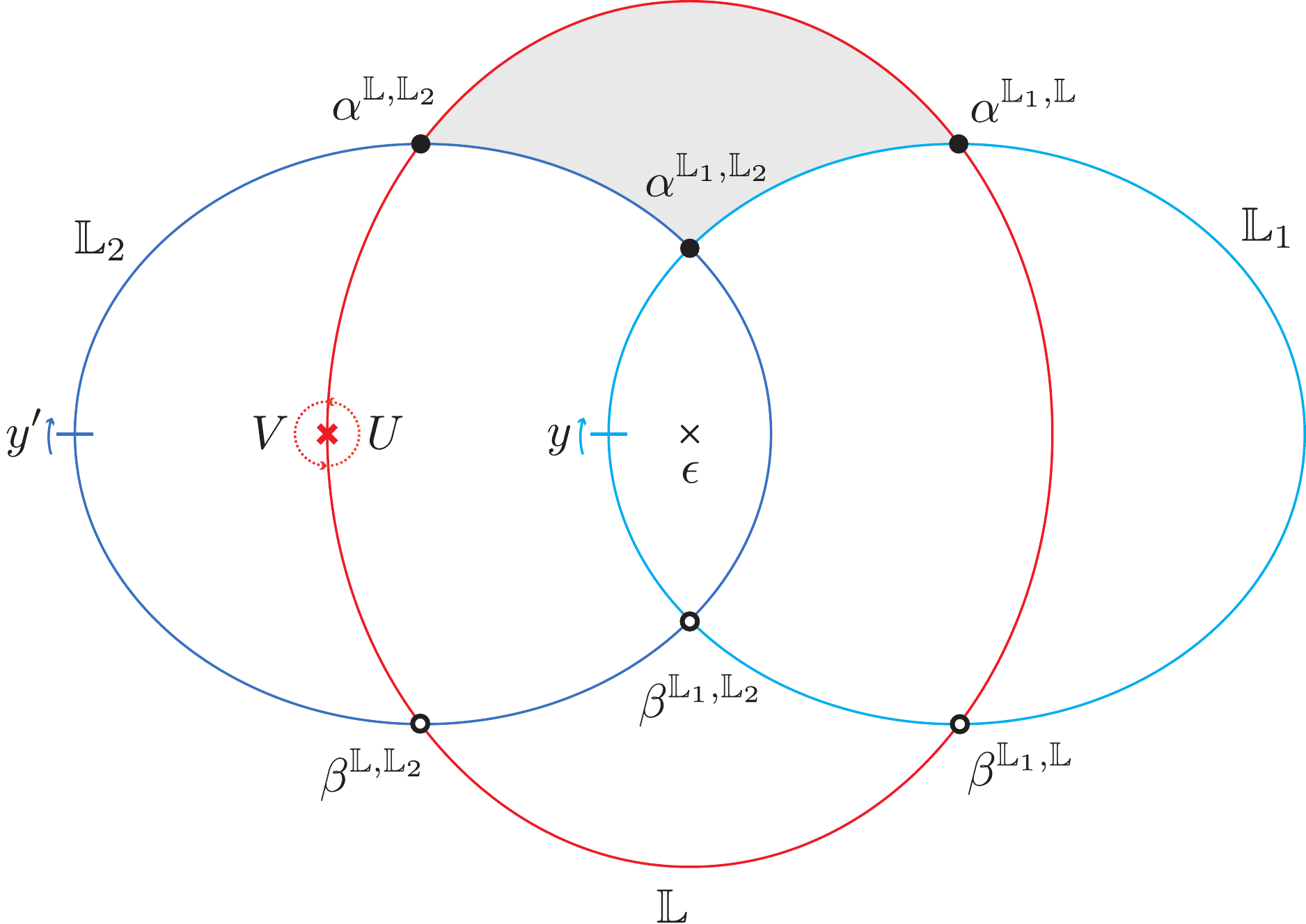}
\caption{Base paths for $\BL_1$,$\BL_2$ and $\BL$}
\label{fig:gluing}
\end{center}
\end{figure}

Now consider the formal deformations $b = uU + vV$ for the immersed sphere $\BL$, where $U,V$ are the degree-one immersed generators.  (See \cite{CHL} or \cite{CHL2} for more details for formal deformation spaces of Lagrangian immersions).  $u,v \in \Lambda_+$ ensures Novikov convergence, namely there are only finitely many terms under every energy level.  First we observe that we can allow $(u,v) \in \Lambda_0 \times \Lambda_+ \cup \Lambda_+ \times \Lambda_0$.  It is because $\BL$ only bounds constant holomorphic polygons whose corners must be in alternating pattern $U,V,\ldots$. Observe that $uv$ still lies in $\Lambda_+$ even if we allow one of $u,v$ in $\Lambda_0$.  Thus we have the following.

\begin{lemma}
The $A_\infty$ operations for $\BL^{uU + vV}$ and other Lagrangians converge for $(u,v) \in \Lambda_0 \times \Lambda_+ \cup \Lambda_+ \times \Lambda_0$.
\end{lemma}

Next, we observe that $\BL$ is unobstructed, namely \begin{equation}\label{eqn:mcfirst}
m_0^b = m_0(1) + m_1(b) + m_2(b,b) + \cdots
\end{equation}
vanishes.  $m_0^b$ has degree $2$ which must be proportional to the point class.  It is merely contributed by constant polygons supported at the self intersection point of $\BL$ whose corners are $U,V,\ldots,U,V$ or $V,U,\ldots,V,U$.  The constant polygons with corners $U,V,\ldots,U,V$ cancel with that with $V,U,\ldots,V,U$ and so $m_0^b=0$.

\begin{lemma}
The formal deformation $b=uU + vV$ of $\BL$ is unobstructed.
\end{lemma}

For instance, the constant bigons contribute
$$ m_2 (uU,vV) = vu[\pt],\qquad m_2(vV,uU) = -uv [\pt],$$
and hence they cancel with each other in \eqref{eqn:mcfirst}. 


\subsection{Gluing deformation spaces}

By analyzing cocycle conditions contributed from holomorphic strips bounded between two of $(\BL_1,\nabla^{x,y})$, $(\BL_2,\nabla^{x',y'})$ and $(\BL, uU+vV)$, we have the following.

\begin{prop}\label{prop:fingl} There are isomorphisms among $(\BL_1,\nabla^{x,y})$, $(\BL_2,\nabla^{x',y'})$ and $(\BL, uU+vV)$ given as follows.
\begin{itemize}
\item $(\BL_1,\nabla^{x,y}) \cong (\BL,uU+vV)$ if and only if
$$x= uv-1,\quad y^{-1}= u  $$ 
where $u,y \in \Lambda_0^\times$, $v \in \Lambda_+$, $x \in -1 + \Lambda_+$.
\item $(\BL_2,\nabla^{x',y'}) \cong (\BL,uU+vV)$ if and only if
$$x'= uv-1,\quad y'=v  $$ 
where $v,y' \in \Lambda_0^\times$, $u \in \Lambda_+$, $x' \in -1 + \Lambda_+$.
\item $(\BL_1,\nabla^{x,y})\cong (\BL_2,\nabla^{x',y'})$ if and only if
$$x= x',\quad y' = y (x+1) $$
where $y,y' \in \Lambda_0^\times$, $x,x' \in k + \Lambda_+$ for $k \in \C^\times - \{-1\}$.
\end{itemize}
\end{prop}

Now we glue the three deformation spaces accordingly.  Namely we take the disjoint union of $(\Lambda_0 \times \Lambda_+) \cup (\Lambda_+ \times \Lambda_0)$ (with coordinates $(u,v)$) and two copies of $\Lambda_0^\times \times \Lambda_0^\times$ (with coordinates $(x,y)$ and $(x',y')$ respectively), and take the quotient according to the above relations. 

To better understand the topology, we should consider pseudo-deformations and pseudo-isomorphisms.  Namely, we take the disjoint union of $\Lambda \times \Lambda \supset (\Lambda_0 \times \Lambda_+) \cup (\Lambda_+ \times \Lambda_0)$ and two copies of $\Lambda^\times \times \Lambda^\times \supset \Lambda_0^\times \times \Lambda_0^\times$, and take the quotient according to the same relations.  The above moduli space is a subset of this.  The advantage of extending by pseudo-deformations is that the gluing regions become open: they are $\Lambda^\times \times \Lambda^\times$ (in coordinates $(x,y)$ or $(x',y')$) between $\BL_i$ and $\BL$, and $(\Lambda^\times - \{-1\}) \times \Lambda^\times$ between $\BL_1$ and $\BL_2$.  We have the quotient topology on the bigger space which restricts to give a topology on the moduli space.

The gluing is illustrated in Figure \ref{fig:glcharts}.  The upper component of the degenerate cone $\{uv=0\}$ is a (partial) compactification of $\{x=-1\} \cong \C^\times$ by adding the point $\{x=-1, y=0\}$, and similar for the lower component.  The resulting moduli space is
$$\{(u,v) \in \Lambda_0 \times \Lambda_0: uv \not\in 1 + \Lambda_+\}.$$

\begin{figure}[h]
\begin{center}
\includegraphics[scale=0.5]{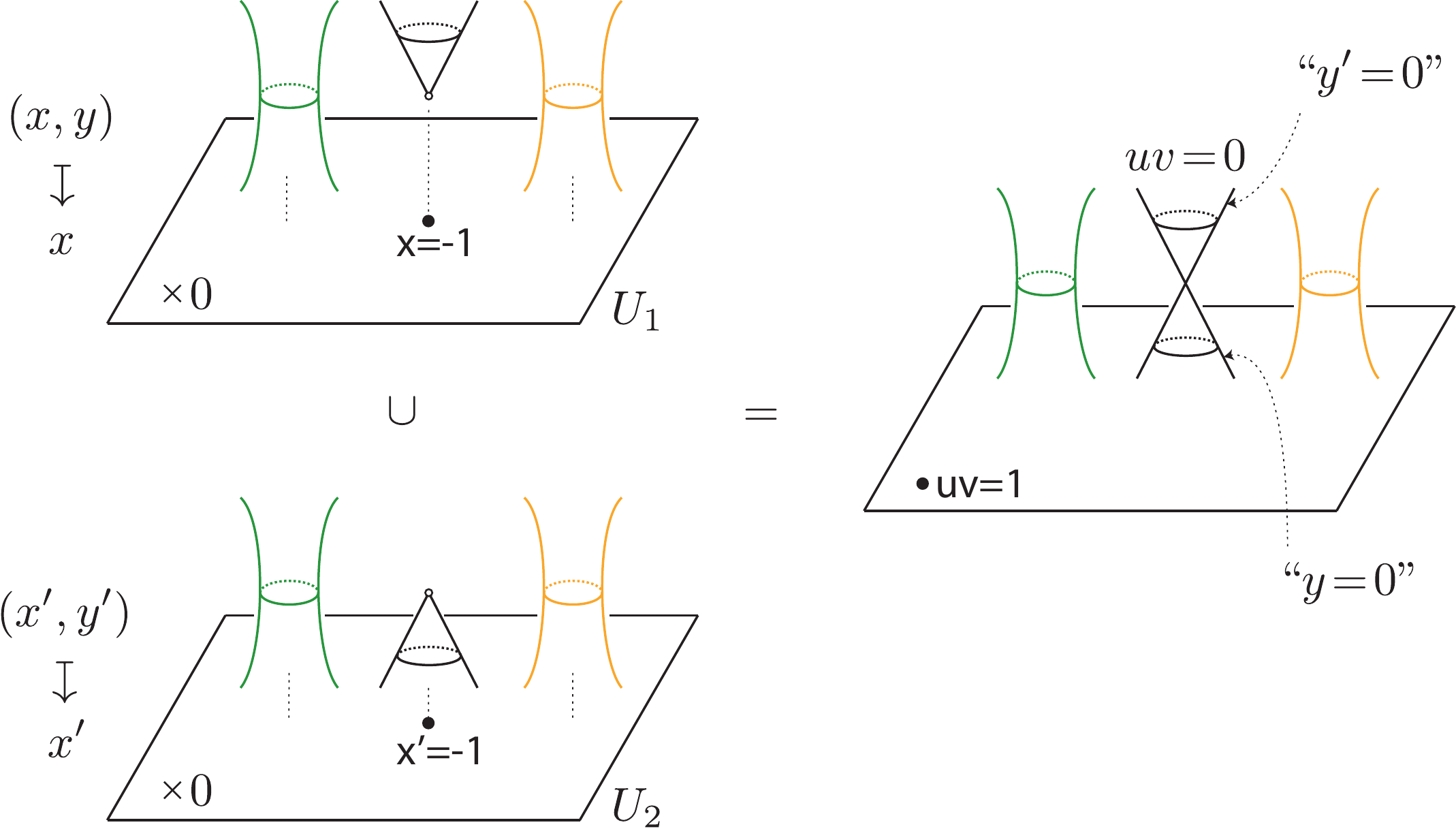}
\caption{Gluing of two charts}
\label{fig:glcharts}
\end{center}
\end{figure}

In the following we sketch the proof of Proposition \ref{prop:fingl} by counting holomorphic strips that contribute to cocycle conditions.

\subsection{Cocycle conditions and isomorphisms}

The relation between $\BL_1$ and $\BL_2$ has been studied in Seidel's lecture notes \cite{Seidel-lect} and also used by the work of Pascaleff-Tonkonog \cite{PT}.  Here we focus on the relation between $\BL$ and $\BL_1$ (and that between $\BL$ and $\BL_2$ is similar).

$\BL_1$ and $\BL$ intersect along two disjoint circles.  See Figure \ref{fig:dalpha}.  We perturb $\BL_1$ by a Hamiltonion so that they intersect transversely at four points, which are denoted by
$$\alpha_0^{\BL_1,\BL}, \alpha_1^{\BL_1,\BL},\beta_1^{\BL_1,\BL},\beta_2^{\BL_1,\BL}.$$
The subscripts denote degrees of the morphisms.

We compute $d=m_1^{\nabla^{x,y},b}$ of the unique degree-zero generator $\alpha_0^{\BL_1,\BL}$ in the Floer complex $CF((\BL_1,\nabla^{x,y}),(\BL,b=uU+vV))$
by counting pearl trajectories.  There are two strips from $\alpha_0^{\BL_1,\BL}$ to $\beta_1^{\BL_1,\BL}$ whose projections under $f$ are the shaded regions in (a) of Figure \ref{fig:dalpha}.  We have chosen $\BL_1$ such that the two strips have the same symplectic area $\Delta$.  We have
\begin{equation}\label{eqn:l1im}
\langle d(\alpha_0^{\BL_1,\BL}),\beta_1^{\BL_1,\BL} \rangle = T^{\Delta}\left( 1-  u y \right).
\end{equation}
The strip on the right side gives $T^\Delta$, while the strip on the left side gives $T^\Delta uy$ because it passes through the immersed point of $\BL$ and the $y$-hypertorus.

\begin{figure}[h]
\begin{center}
\includegraphics[scale=0.4]{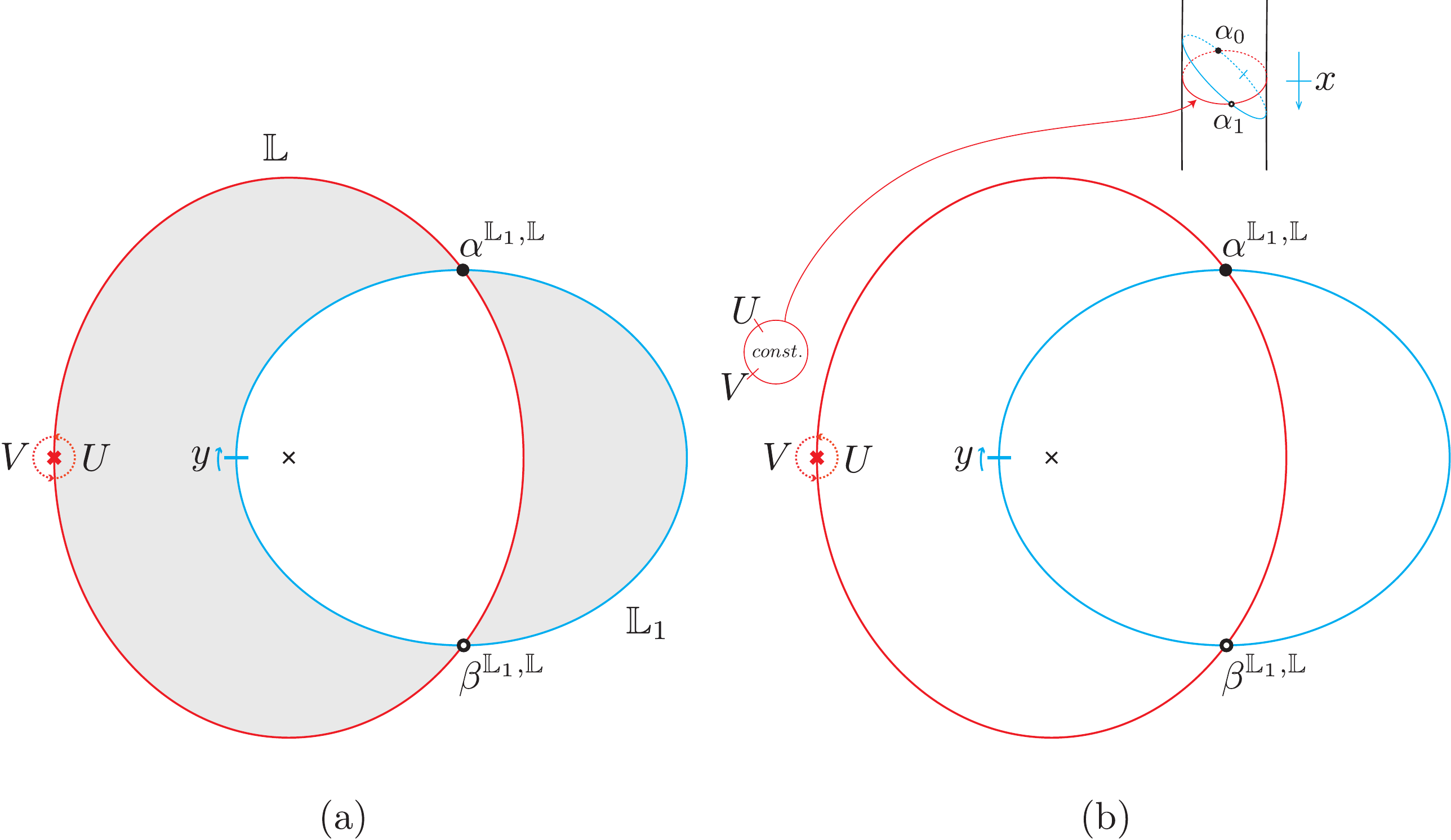}
\caption{Pearl trajectories contributing to $d (\alpha^{\BL_1,\BL}  )$}
\label{fig:dalpha}
\end{center}
\end{figure}

For $\left\langle d(\alpha_0^{\BL_1,\BL}), \alpha_1^{\BL_1,\BL} \right\rangle$, we also have a pair of strips in $f^{-1} (\alpha^{\BL_1,\BL})$ contributing to it.  One of the strips gives $\pm 1$, and the other gives $\pm x$ since it passes through the $x$-hypertorus.   Additionally we have pearl trajectories consisting of a strip, Morse flow lines in (the normalization of) $\bL$, and constant polygons with corners $U,V,\ldots,U,V$.  The pearl trajectories form power series in $uv$.  Hence
$$\left\langle d(\alpha_0^{\BL_1,\BL}), \alpha_1^{\BL_1,\BL} \right\rangle = x h(uv) + g(uv)$$
where $g,h$ are power series with leading term $\pm 1$.  In summary
$$ d(\alpha_0^{\BL_1,\BL}) = T^{\Delta}\left( 1- uy \right)\beta_1^{\BL_1,\BL} + (x h(uv) + g(uv)) \alpha_1^{\BL_1,\BL}.$$
The cocycle condition gives
\begin{equation}\label{eqn:coord1im}
\left\{
\begin{array}{l}
x + H(uv) = 0\\
y^{-1} = u
\end{array}\right.
\end{equation}
where $H=g/h$.
One can check that the degree zero morphism in $CF(\BL,\BL_1)$ (which is the one dual to $\beta_2^{\BL_1,\BL}$) is also $d$-closed if and only if the same condition is satisfied. Moreover, $\alpha_0^{\BL_1,\BL}$ gives an isomorphism under the condition \ref{eqn:coord1im}.

$H(uv)$ can be derived from the relations between $\BL,\BL_1,\BL_2$.

\begin{lemma} 
We have $H(uv) = 1-uv$ and hence $x = uv-1$. 
\end{lemma}

\begin{proof} 
We have chains of isomorphisms
$$(\BL_1,(x,y)) \stackrel{\alpha_0^{\BL_1,\BL}}{\longrightarrow} (\BL, uU+ vV) \stackrel{\alpha_0^{\BL,\BL_2}}{\longrightarrow} (\BL_2,(x',y')).$$
The composition of these two isomorphisms is given by
$$m_2( \alpha_0^{\BL_1,\BL}, \alpha_0^{\BL,\BL_2}) = T^{\Delta'} \alpha_0^{\BL_1,\BL_2} $$
where $\Delta'$ is the symplectic area of the triangle projecting down to the shaded region in Figure \ref{fig:gluing}. This implies that $\alpha_0^{\BL_1,\BL_2}$ is also an isomorphism. Therefore, coordinate changes among three objects must be compatible (over the triple interesction). Using
$$x'=x,\quad y' = y (x+1)$$
we get
$$ 1-H(uv) = x+1 =  \frac{y'}{y} = \frac{v}{u^{-1}} = uv.$$
\end{proof}

In conclusion, the moduli space of smooth Lagrangian tori has a puncture $u=v=0$.  We can glue in the deformation space of the immersed sphere to partial compactify and fill in this puncture.   It also applies to generic singular fibers of a Lagrangian fibration in higher dimensions by taking a direct product with a torus factor.  

In general situations, the deformation space of a singular Lagrangian is noncommutative, and our gluing technique will produce a noncommutative space.  It is already manifested in the one-dimensional example in Figure \ref{fig:flop}.  This is very interesting and we will further investigate it in future works.



\bibliographystyle{amsalpha}
\bibliography{geometry}

%
%
%
%
%
%
%

\end{document}